\documentclass[10pt,a4paper,twoside]{amsart}

\usepackage{graphicx}
\usepackage{a4wide}

\newtheorem{definition}{Definition}[section]
\newtheorem{theorem}{Theorem}[section]
\newtheorem{lemma}{Lemma}[section]

\newtheorem{remark}{Remark}[section]
\newtheorem*{maintheorem*}{Main Theorem}
\allowdisplaybreaks
\numberwithin{equation}{section}

\newcommand{\norm}[1]{\left\| #1 \right\|}

\newcommand{\eps}{\varepsilon}

\newcommand{\eb}{{\eps,\beta}}
\newcommand{\ueb}{u_\eb}

\newcommand{\pt}{\partial_t}

\newcommand{\px}{\partial_x }
\newcommand{\pxx}{\partial_{xx}^2}
\newcommand{\pxxx}{\partial_{xxx}^3}
\newcommand{\ptxxx}{\partial_{txxx}^4}
\newcommand{\pxxxx}{\partial_{xxxx}^4}

\newcommand{\ptxxxx}{\partial_{txxxx}^5}
\newcommand{\ptx}{\partial_{tx}^2}
\newcommand{\ptxx}{\partial_{txx}^3}

\renewcommand{\i}{\ifmmode\mathit{\mathchar"7010 }\else\char"10 \fi}
\renewcommand{\j}{\ifmmode\mathit{\mathchar"7011 }\else\char"11 \fi}
\newcommand{\R}{\mathbb{R}}
\newcommand{\N}{\mathbb{N}}

{%

\begin{enumerate}}%
{\end{enumerate}}

{%

\begin{enumerate}}%
{\end{enumerate}}

\begin{document}\large

\title[A Singular limit problem of  Rosenau-KdV type]{A singular limit problem for conservation laws \\ related to the Rosenau-Korteweg-de Vries equation}
\author[G. M. Coclite and L. di Ruvo]{Giuseppe Maria Coclite and Lorenzo di Ruvo}
\address[Giuseppe Maria Coclite and Lorenzo di Ruvo]
{\newline Department of Mathematics,   University of Bari, via E. Orabona 4, 70125 Bari,   Italy}
\email[]{giuseppemaria.coclite@uniba.it, lorenzo.diruvo@uniba.it}
\urladdr{http://www.dm.uniba.it/Members/coclitegm/}

\keywords{Singular limit, compensated compactness, Rosenau-KdV-equation, entropy condition.}

\subjclass[2000]{35G25, 35L65, 35L05}


\thanks{The authors are members of the Gruppo Nazionale per l'Analisi Matematica, la Probabilit\`a e le loro Applicazioni (GNAMPA) of the Istituto Nazionale di Alta Matematica (INdAM)}

\begin{abstract}
We consider the  Rosenau-Korteweg-de Vries equation, which contains nonlinear dispersive effects. We prove
that as the diffusion parameter tends to zero, the solutions of the dispersive equation converge to discontinuous weak
solutions of the Burgers equation.
The proof relies on deriving suitable a priori estimates together with an application of the compensated compactness method in the $L^p$ setting.
\end{abstract}

\maketitle


\section{Introduction}\label{sec:intro}
Dynamics of shallow water waves that is observed along lake shores and beaches has been a research area for the past few decades in
oceanography (see \cite{AB,ZZZC}). There are several models proposed in this context: Boussinesq equation, 
Peregrine equation, regularized long wave (RLW) equation, Kawahara equation, Benjamin-Bona-Mahoney equation, Bona-Chen equation etc. 
These models are
derived from first principles under various different hypothesis and approximations. They are all well studied and very well understood.

In this context, there is also the Korteweg-de Vries equation
\begin{equation}
\label{eq:ZIU7}
\pt u +\px u^2 +\beta\pxxx u=0.
\end{equation}
Observe that, if we send $\beta\to0$ in \eqref{eq:ZIU7}, we pass from \eqref{eq:ZIU7} to the Burgers equation
\begin{equation}
\label{eq:BU}
\pt u +\px u^2 =0.
\end{equation}
In cite \cite{LN,SC}, the convergence of the solution of \eqref{eq:ZIU7} to the unique entropy solution of \eqref{eq:BU} is proven, under the assumption
\begin{equation}
\label{eq:assu-1}
u_{0}\in L^2(\R)\cap L^4(\R), \quad \beta=o\left(\eps^2\right).
\end{equation}
\cite[Appendixes $A$ and $B$]{Cd6} show that it is possible to obtain the same result of convergence, under the following assumptions
\begin{equation}
\label{eq:assu-2}
\begin{split}
& u_{0}\in L^2(\R),\quad -\infty<\int_{\R}u_{0}(x) dx<\infty,  \quad \beta=o\left(\eps^3\right),\\
& u_{0}\in L^2(\R),\quad \beta=o\left(\eps^4\right).
\end{split}
\end{equation}
One generalization of \eqref{eq:ZIU7} is the Ostrovsky equation (see \cite{O}):
\begin{equation}
\label{eq:OHbeta}
\px(\pt u+\px u^2-\beta \pxxx u)=\gamma u, \quad \beta,\gamma\in\R.
\end{equation}
\eqref{eq:OHbeta} describes small-amplitude long waves in a rotating fluid of a finite depth by the additional term induced by
the Coriolis force.
If we send $\beta\to 0$ in \eqref{eq:OHbeta}, we pass from \eqref{eq:OHbeta} to the Ostrovsky-Hunter equation (see \cite{B}).
\begin{equation}
\label{eq:OH}
\px(\pt u+\px u^2)=\gamma u,\qquad t>0, \quad x\in\R.
\end{equation}
In \cite{Cd1,CdK,dR}, the wellposedness of the entropy solutions of \eqref{eq:OH} is proven, in the sense of the following definition:
\begin{definition}
\label{def:sol}
We say that $u\in  L^{\infty}((0,T)\times\R),\,T>0,$ is an entropy solution of the initial
value problem \eqref{eq:OH} if
\begin{itemize}
\item[$i$)] $u$ is a distributional solution of \eqref{eq:OH};
\item[$ii$)] for every convex function $\eta\in  C^2(\R)$ the
entropy inequality
\begin{equation}
\label{eq:OHentropy}
\pt \eta(u)+ \px q(u)-\gamma\eta'(u) P\le 0, \qquad     q(u)=\int^u f'(\xi) \eta'(\xi)\, d\xi,
\end{equation}
holds in the sense of distributions in $(0,\infty)\times\R$.
\end{itemize}
\end{definition}
Under the assumption \eqref{eq:assu-1}, in \cite{Cd2}, the convergence of the solutions of \eqref{eq:OHbeta} to the unique entropy solution of \eqref{eq:OH} is proven.

The dynamics of dispersive shallow water waves, on the other hand, is captured with slightly different models, like the
 Rosenau-Kawahara equation and the Rosenau-KdV-RLW equation \cite{BTL,EMTYB,HXH,LB,RAB}.

The Rosenau-Korteweg-de Vries-RLW equation is following one:
\begin{equation}
\label{eq:RKV-1}
\pt u +a\px u +k\px u^{n}+b_1\pxxx u +b_2\ptxx u + c\ptxxxx u=0,\quad a,\,k,\,b_1,\,b_2,\,c\in\R.
\end{equation}
Here $u(t,x)$ is the nonlinear wave profile. The first term is the linear evolution one, while $a$ is the advection or drifting coefficient. The two dispersion 
coefficients are $b_1$ and $b_2$. The higher order dispersion coefficient is $c$, while the coefficient of nonlinearity is $k$ where $n$ is nonlinearity parameter. These are all known and given parameters.

In \cite{RAB}, the authors analyzed \eqref{eq:RKV-1}. They got solitary waves, shock waves and singular solitons along with conservation laws.

Considering the  $n=2,\, a=0,\, k=1,\, b_1=1,\, b_2=-1,\, c=1$:
\begin{equation}
\label{eq:RKV-23}
\pt u +\px u^2 +\pxxx u -\ptxx u +\ptxxxx u=0.
\end{equation}
If $n=2, \, a=0,\, k=1,\, b_1=0,\, b_2=-1,\, c=1$, \eqref{eq:RKV-1} reads
\begin{equation}
\label{eq:RKV-30}
\pt u +\px u^2 -\ptxx u +\ptxxxx u=0,
\end{equation}
which is known as Rosenau-RLW equation.

Arguing in \cite{CdREM}, we re-scale the equations as
follows
\begin{align}
\label{eq:T1}
\pt u +\px u^2 +\beta\pxxx u -\beta\ptxx u +\beta^2\ptxxxx u&=0,\\
\label{eq:T2}
\pt u +\px u^2 -\beta\ptxx u +\beta^2\ptxxxx\ueb&=0,
\end{align}
where $\beta$ is the diffusion parameter.

In \cite{Cd5}, the authors proved that the solutions of \eqref{eq:T1} and \eqref{eq:T2} converge to the unique entropy solution of \eqref{eq:BU}, under the assumptions
\begin{equation}
\label{eq:uo-l4}
u_{0}\in L^2(\R)\cap L^4(\R), \quad \beta=\mathbf{\mathcal{O}}\left(\eps^4\right).
\end{equation}

\eqref{eq:ZIU7} has also been used in very wide applications and undergone research which can be used to describe wave propagation and spread interaction (see \cite{Ba,CM,OK,ZUZO}).

In the study of the dynamics of dense discrete systems, the case of wave-wave and wave-wall interactions cannot be described using \eqref{eq:ZIU7}. To overcome
this shortcoming of \eqref{eq:ZIU7}, Rosenau proposed the following equation (see \cite{Ro1,Ro2}):
\begin{equation}
\label{eq:ROUS1}
\pt u + \px u^2  + \ptxxxx u =0,
\end{equation}
which is also obtained by \eqref{eq:RKV-1}, taking $n=2,\, a=0,\, k=1,\, b_1=0,\, b_2=0,\, c=1$.

The existence and the uniqueness of the solution for \eqref{eq:ROUS1} is proved in \cite{P}, but it is difficult to find the analytical
solution for \eqref{eq:ROUS1}. Therefore, much work has been done on the numerical methods for \eqref{eq:ROUS1} (see \cite{CH1,CHH,CHP,KL,MPC,OAAK}).

On the other hand, for the further consideration of the nonlinear wave, the viscous term $\pxxx u$ needs to be included (see \cite{Z}). In this case, \eqref{eq:ROUS1} reads
\begin{equation}
\label{eq:ROUS2}
\pt u +  + \px u^2 +\pxxx u + \ptxxxx u =0,
\end{equation}
which is known as the Rosenau-Korteweg-de Vries (KdV) equation, and is also obtianed by \eqref{eq:RKV-1}, taking  $n=2,\, a=0,\, k=1,\, b_1=1,\, b_2=0,\, c=1$.

In \cite{Z}, the author discussed the solitary wave solutions and \eqref{eq:ROUS2}. In \cite{HXH}, a conservative linear finite difference scheme for the numerical solution for an initial-boundary value problem of Rosenau-KdV equation is considered.
In \cite{E,RTB}, authors discussed the solitary solutions for \eqref{eq:ROUS2} with usual solitary ansatz method. The authors also gave the two invariants for \eqref{eq:ROUS2}. In particular, in \cite{RTB}, the authors not only studied the two types of soliton solution, one is solitary wave solution and the other is singular soliton. In \cite{ZZ}, the authors proposed an average linear finite difference scheme for the numerical solution of the initial-boundary value problem for  \eqref{eq:ROUS2}.

Consider \eqref{eq:ROUS1}. Arguing as \cite{CdREM}, we re-scale the equations as follows
\begin{equation}
\label{eq:P90}
\pt u +\px u^2 +\beta^2\ptxxxx\ueb=0.
\end{equation}
In \cite{Cd6}, the authors proved that the solutions of \eqref{eq:P90} converge to the unique entropy solution of \eqref{eq:BU}, choosing the initial datum in two different ways. The first one is:
\begin{equation}
\label{eq:uo-l2}
u_{0}\in L^2(\R), \quad \beta=o\left(\eps^4\right).
\end{equation}
The second choice is given by \eqref{eq:uo-l4}.

In this paper, we analyze \eqref{eq:ROUS2}. Arguing as \cite{CdREM}, we re-scale the equations as follows
\begin{equation}
\label{eq:RKV33}
\pt u+ \px u^2+\beta\pxxx u+\beta^2\ptxxxx u=0.
\end{equation}

We are interested in the no high frequency limit,  we send $\beta\to 0$ in \eqref{eq:RKV33}. In this way we pass from \eqref{eq:RKV33} to \eqref{eq:BU}.
We prove that, as $\beta\to0$, the solutions of \eqref{eq:RKV33} to the unique entropy solution of \eqref{eq:BU}.
In other to do this,  we can choose the initial datum and $\beta$ in two different ways.
Following \cite[Theorem $7.1$]{CRS}, the first choice is given by \eqref{eq:uo-l2} (see Theorem \ref{th:main-1}).
Since $\norm{\cdot}_{L^4}$ is a conserved quantity for \eqref{eq:RKV33},  the second  choice is given by \eqref{eq:uo-l4} (see Theorem \ref{th:main-13}).
It is interesting to observe that, while the summability on the initial datum in \eqref{eq:uo-l4} is greater than the one of \eqref{eq:uo-l2}, 
the assumption on $\beta$ in \eqref{eq:uo-l4} is weaker than the one in \eqref{eq:uo-l2}.

From the mathematical point of view, the two assumptions require two different arguments for the $L^{\infty}-$estimate (see Lemmas \ref{lm:50} and \ref{lm:562}). Indeed, the proof of Lemma \ref{lm:50}, under the assumption \eqref{eq:uo-l2}, is more technical than the one of Lemma \ref{lm:562}.
Moreover, due to the presence of the third order term, Lemmas \ref{lm:50} and \ref{lm:t3}  is finer than \cite[Lemmas $2.2$ and $3.2$]{Cd6}. Indeed, with respect to \cite[Lemma $2.2$]{Cd6}, in Lemma \ref{lm:50} we need to prove the existence of two positive constants, while, 
with respect to \cite[Lemma $3.2$]{Cd6}, in Lemma \ref{lm:t3} we need to prove the existence of four positive constants.

The paper is is organized in four sections. In Section \ref{sec:Ro1}, we prove the convergence of \eqref{eq:RKV33} to \eqref{eq:BU} in $L^{p}$ setting, with $1\le p< 2$. In Section \ref{sec:D1}, we prove the convergence of \eqref{eq:RKV33} to \eqref{eq:BU} in $L^{p}$ setting, with $1\le p< 4$. The Section \ref{appen1} is an appendix where we prove that the solutions of the the Benjamin-Bona-Mahony equation converge to discontinuous weak solutions of \eqref{eq:BU} in in $L^{p}$ setting, with $1\le p< 2$.

\section{The Rosenau-KdV-equation: $u_0\in L^2(\R)$.}\label{sec:Ro1}
In this section, we consider \eqref{eq:RKV33}, and assume \eqref{eq:uo-l2} on the initial datum.
We study the dispersion-diffusion limit for \eqref{eq:RKV33}. Therefore, we fix two small numbers
$0 < \eps,\,\beta < 1$ and consider the following fifth order approximation
\begin{equation}
\label{eq:Ro-eps-beta}
\begin{cases}
\pt\ueb+ \px \ueb^2 +\beta\pxxx\ueb +\beta^2\ptxxxx\ueb=\eps\pxx\ueb, &\qquad t>0, \ x\in\R ,\\
\ueb(0,x)=u_{\eps,\beta,0}(x), &\qquad x\in\R,
\end{cases}
\end{equation}
where $u_{\eps,\beta,0}$ is a $C^\infty$ approximation of $u_{0}$ such that
\begin{equation}
\begin{split}
\label{eq:u0eps-1}
&u_{\eps,\,\beta,\,0} \to u_{0} \quad  \textrm{in $L^{p}_{loc}(\R)$, $1\le p < 2$, as $\eps,\,\beta \to 0$,}\\
&\norm{u_{\eps,\beta, 0}}^2_{L^2(\R)}+(\beta^{\frac{1}{2}}+ \eps^2) \norm{\px u_{\eps,\beta,0}}^2_{L^2(\R)}\le C_0,\quad \eps,\beta >0,\\
&\left(\beta^2 +\beta\eps^2\right) \norm{\pxx u_{\eps,\beta,0}}^2_{L^2(\R)} +\beta^{\frac{5}{2}}\norm{\pxxx u_{\eps,\beta,0}}^2_{L^2(\R)}\le C_0,\quad \eps,\beta >0,
\end{split}
\end{equation}
and $C_0$ is a constant independent on $\eps$ and $\beta$.

The main result of this section is the following theorem.
\begin{theorem}
\label{th:main-1}
Assume that \eqref{eq:uo-l2} and  \eqref{eq:u0eps-1} hold. Fix $T>0$,
if
\begin{equation}
\label{eq:beta-eps-2}
\beta=\mathbf{\mathcal{O}}\left(\eps^4\right),
\end{equation}
then, there exist two sequences $\{\eps_{n}\}_{n\in\N}$, $\{\beta_{n}\}_{n\in\N}$, with $\eps_n, \beta_n \to 0$, and a limit function
\begin{equation*}
u\in L^{\infty}((0,T); L^2(\R)),
\end{equation*}
such that
\begin{itemize}
\item[$i)$] $u_{\eps_n, \beta_n}\to u$  strongly in $L^{p}_{loc}(\R^{+}\times\R)$, for each $1\le p <2$,
\item[$ii)$] $u$ is a distributional solution of \eqref{eq:BU}.
\end{itemize}
Moreover, if
\begin{equation}
\label{eq:beta-eps-4}
\beta=o\left(\eps^{4}\right),
\end{equation}
\begin{itemize}
\item[$iii)$] $u$ is the unique entropy solution of \eqref{eq:BU}.
\end{itemize}
\end{theorem}
Let us prove some a priori estimates on $\ueb$, denoting with $C_0$ the constants which depend only on the initial data.
\begin{lemma}\label{lm:38}
For each $t>0$,
\begin{equation}
\label{eq:l-2-u1}
\norm{\ueb(t,\cdot)}^2_{L^2(\R)}+\beta^2\norm{\pxx\ueb(t,\cdot)}^2_{L^2(\R)} + 2\eps\int_{0}^{t}\norm{\px\ueb(t,\cdot)}^2_{L^2(\R)}\le C_0.
\end{equation}
\end{lemma}
\begin{proof}
We begin by observing that
\begin{equation*}
\int_{\R}\ueb\pxxx\ueb dx =0.
\end{equation*}
Therefore, arguing as \cite[Lemma $2.1$]{Cd6}, we have \eqref{eq:l-2-u1}.
\end{proof}
\begin{lemma}\label{lm:50}
Fix $T>0$. Assume \eqref{eq:beta-eps-2} holds. There exists $C_0>0$, independent on $\eps,\,\beta$ such that
\begin{equation}
\label{eq:u-infty-3}
\norm{\ueb}_{L^{\infty}((0,T)\times\R)}\le C_0\beta^{-\frac{1}{4}}.
\end{equation}
Moreover,
\begin{itemize}
\item[$i)$] the families  $\{\beta^{\frac{1}{2}}\px\ueb\}_{\eps,\,\beta},\,\{\beta^{\frac{1}{4}}\eps\px\ueb\}_{\eps,\,\beta},\,\{\beta^{\frac{3}{4}}\eps\pxx\ueb\}_{\eps,\,\beta}, \{\beta^{\frac{3}{2}}\pxxx\ueb\}_{\eps,\,\beta},$\\ are bounded in $L^{\infty}((0,T);L^{2}(\R))$;
\item[$ii)$] the families $\{\beta^{\frac{3}{4}}\eps^{\frac{1}{2}}\ptx\ueb\}_{\eps,\,\beta},$ $\{\beta^{\frac{7}{4}}\eps^{\frac{1}{2}}\ptxxx\ueb\}_{\eps,\,\beta},$
$\{\beta^{\frac{1}{4}}\eps\pt\ueb\}_{\eps,\,\beta},$ \\ $\{\beta^{\frac{5}{4}}\eps^{\frac{1}{2}}\ptxx\ueb\}_{\eps,\,\beta}$,
$\{\beta^{\frac{1}{2}}\eps^{\frac{1}{2}}\pxx\ueb\}_{\eps,\,\beta}$ are bounded in $L^2((0,T)\times\R)$.
\end{itemize}
\end{lemma}
\begin{proof}
Let $0<t<T$. Let $A,\,B$ be some positive constants which will be specified later.  Multiplying \eqref{eq:Ro-eps-beta} by $-\beta^{\frac{1}{2}}\pxx\ueb -A\beta\eps \ptxx\ueb +B\eps \pt\ueb$, we have
\begin{equation}
\label{eq:p456}
\begin{split}
&\left(-\beta^{\frac{1}{2}}\pxx\ueb -A\beta\eps \ptxx\ueb +B\eps \pt\ueb\right)\pt\ueb\\
&\qquad\quad +2\left(-\beta^{\frac{1}{2}}\pxx\ueb -A\beta\eps \ptxx\ueb+B\eps \pt\ueb\right)\ueb\px\ueb\\
&\qquad\quad\beta \left(-\beta^{\frac{1}{2}}\pxx\ueb -A\beta\eps \ptxx\ueb+ B\eps\pt\ueb\right)\pxxx\ueb\\
&\qquad\quad +\beta^2\left(-\beta^{\frac{1}{2}}\pxx\ueb - A\beta\eps \ptxx\ueb +B\eps \pt\ueb\right)\ptxxxx\ueb\\
&\qquad=\eps\left(-\beta^{\frac{1}{2}}\pxx\ueb -A\beta\eps \ptxx\ueb +B\eps \pt\ueb\right)\pxx\ueb.
\end{split}
\end{equation}
We observe that
\begin{equation}
\label{eq:L1}
\begin{split}
&\int_{\R}\left(-\beta^{\frac{1}{2}}\pxx\ueb -A\beta\eps \ptxx\ueb +B\eps \pt\ueb\right)\pt\ueb dx\\
&\qquad= \frac{\beta^{\frac{1}{2}}}{2}\frac{d}{dt}\norm{\px\ueb(t,\cdot)}^2_{L^2(\R)}+ \beta\eps A\norm{\ptx\ueb(t,\cdot)}^2_{L^2(\R)}\\
&\qquad\quad +B\eps\norm{\pt\ueb(t,\cdot)}^2_{L^2(\R)}.
\end{split}
\end{equation}
Since
\begin{align*}
2\int_{\R}&\left(-\beta^{\frac{1}{2}}\pxx\ueb -A\beta\eps \ptxx\ueb +B\eps \pt\ueb\right)\ueb\px\ueb dx\\
=& -2\beta^{\frac{1}{2}}\int_{\R}\ueb\px\ueb\pxx\ueb dx -2A\beta\eps\int_{\R}\ueb\px\ueb\ptxx\ueb dx \\
&+2B\eps\int_{\R}\ueb\px\ueb\pt\ueb dx,\\
\beta \int_{\R}&\left(-\beta^{\frac{1}{2}}\pxx\ueb -A\beta\eps \ptxx\ueb+ B\eps\pt\ueb\right)\pxxx\ueb dx \\
=& A\beta^2\eps\int_{\R}\pxx\ueb\ptxxx\ueb dx + B\beta\eps \int_{\R} \px\ueb\ptxx\ueb dx,\\
\beta^2\int_{\R}&\left(-\beta^{\frac{1}{2}}\pxx\ueb - A\beta\eps \ptxx\ueb +B\eps \pt\ueb\right)\ptxxxx\ueb dx\\
=& \frac{\beta^{\frac{5}{2}}}{2}\frac{d}{dt}\norm{\pxxx\ueb(t,\cdot)}^2_{L^2(\R)} + A\beta^3\eps\norm{\ptxxx\ueb(t,\cdot)}^2_{L^2(\R)}\\
&+B\beta^2\eps\norm{\ptxx\ueb(t,\cdot)}^2_{L^2(\R)},\\
\eps\int_{\R}&\left(-\beta^{\frac{1}{2}}\pxx\ueb -A\beta\eps \ptxx\ueb +B\eps \pt\ueb\right)\pxx\ueb dx\\
=& -\beta^{\frac{1}{2}}\eps\norm{\pxx\ueb(t,\cdot)}^2_{L^2(\R)} -\frac{A\beta\eps^2}{2}\frac{d}{dt}\norm{\pxx\ueb(t,\cdot)}^2_{L^2(\R)}\\
&\qquad\quad -\frac{B\eps^2}{2}\frac{d}{dt}\norm{\px\ueb(t,\cdot)}^2_{L^2(\R)},
\end{align*}
an integration on $\R$ of \eqref{eq:L1} gives
\begin{equation}
\label{eq:L6}
\begin{split}
&\frac{d}{dt}\left(\frac{\beta^{\frac{1}{2}} + B\eps^2}{2}\norm{\px\ueb(t,\cdot)}^2_{L^2(\R)}+ \frac{A\beta\eps^2}{2} \norm{\pxx\ueb(t,\cdot)}^2_{L^2(\R)} \right)\\
&\qquad\quad + \frac{\beta^{\frac{5}{2}}}{2}\frac{d}{dt}\norm{\pxxx\ueb(t,\cdot)}^2_{L^2(\R)} + \beta\eps A\norm{\ptx\ueb(t,\cdot)}^2_{L^2(\R)}\\
&\qquad\quad +B\eps\norm{\pt\ueb(t,\cdot)}^2_{L^2(\R)} + A\beta^3\eps\norm{\ptxxx\ueb(t,\cdot)}^2_{L^2(\R)}\\
&\qquad\quad +B\beta^2\eps\norm{\ptxx\ueb(t,\cdot)}^2_{L^2(\R)}+ \beta^{\frac{1}{2}}\eps\norm{\pxx\ueb(t,\cdot)}^2_{L^2(\R)}\\
&\qquad= 2\beta^{\frac{1}{2}}\int_{\R}\ueb\px\ueb\pxx\ueb dx +2A\beta\eps\int_{\R}\ueb\px\ueb\ptxx\ueb dx \\
&\qquad\quad -2B\eps\int_{\R}\ueb\px\ueb\pt\ueb dx - A\beta^2\eps\int_{\R}\pxx\ueb\ptxxx\ueb dx\\
&\qquad\quad - B\beta\eps \int_{\R} \px\ueb\ptxx\ueb dx.
\end{split}
\end{equation}
Using \eqref{eq:u0eps-1}, $0<\beta<1$, and the Young inequality,
\begin{align*}
2\beta^{\frac{1}{2}}&\int_{\R}\vert\ueb\px\ueb\vert\vert\pxx\ueb\vert dx= \beta^{\frac{1}{2}}\int_{\R}
\left\vert\frac{2\ueb\px\ueb}{\eps^{\frac{1}{2}}}\right\vert\left\vert \eps^{\frac{1}{2}}\pxx\ueb(t,\cdot)\right\vert dx \\
\le& \frac{2\beta^{\frac{1}{2}}}{\eps}\int_{\R}\ueb^2(\px\ueb)^2 dx +\frac{\beta^{\frac{1}{2}}\eps}{2}\norm{\pxx\ueb(t,\cdot)}^2_{L^2(\R)}\\
\le& C_{0}\eps\norm{\ueb}^2_{L^2((0,T)\times\R)}\norm{\px\ueb(t,\cdot)}^2_{L^2(\R)} +\frac{\beta^{\frac{1}{2}}\eps}{2}\norm{\pxx\ueb(t,\cdot)}^2_{L^2(\R)},\\
2A\beta\eps\int_{\R}&\vert\ueb\px\ueb\vert\vert\ptxx\ueb\vert dx= 
\eps \int_{\R}\left\vert \frac{2A\ueb\px\ueb}{\sqrt{B}}\right\vert\left\vert\sqrt{B}\beta\ptxx\ueb\right\vert dx\\
\le& \frac{2A^2\eps}{B}\int_{\R}\ueb^2(\px\ueb)^2 dx + \frac{B\beta^2\eps}{2}\norm{\ptxx\ueb(t,\cdot)}^2_{L^2(\R)}\\
\le&\frac{2A^2\eps}{B}\norm{\ueb}^2_{L^2((0,T)\times\R)}\norm{\px\ueb(t,\cdot)}^2_{L^2(\R)} + \frac{B\beta^2\eps}{2}\norm{\ptxx\ueb(t,\cdot)}^2_{L^2(\R)},\\
2B\eps\int_{\R}&\vert\ueb\px\ueb\vert\vert\pt\ueb\vert dx = B\eps\int_{\R} \left\vert 2\ueb\px\ueb\right\vert \left\vert \pt\ueb\right\vert dx \\
\le& 2B\eps\int_{\R}\ueb^2(\px\ueb)^2 dx +\frac{B\eps}{2}\norm{\pt\ueb(t,\cdot)}^2_{L^2(\R)}\\
\le& 2B\eps\norm{\ueb}^2_{L^2((0,T)\times\R)}\norm{\px\ueb(t,\cdot)}^2_{L^2(\R)}+\frac{B\eps}{2}\norm{\pt\ueb(t,\cdot)}^2_{L^2(\R)},\\
A\beta^2\eps\int_{\R}&\vert\pxx\ueb\vert\vert\ptxxx\ueb\vert dx 
= A\eps\int_{\R}\left\vert\beta^{\frac{1}{2}}\pxx\ueb\right\vert\left\vert\beta^{\frac{3}{2}}\ptxxx\ueb \right\vert dx \\
 \le& \frac{A\beta\eps}{2}\norm{\pxx\ueb(t,\cdot)}^2_{L^2(\R)} + \frac{A\beta^3\eps}{2}\norm{\ptxxx\ueb(t,\cdot)}^2_{L^2(\R)}\\
\le& \frac{A\beta^{\frac{1}{2}}\eps}{2}\norm{\pxx\ueb(t,\cdot)}^2_{L^2(\R)} + \frac{A\beta^3\eps}{2}\norm{\ptxxx\ueb(t,\cdot)}^2_{L^2(\R)},\\
B\beta\eps \int_{\R}& \px\ueb\ptxx\ueb dx=\eps\int_{\R}\left\vert\px\ueb\right \vert\left\vert B\beta\ptxx\ueb\right\vert dx\\
\le& \frac{\eps}{2}\norm{\px\ueb(t,\cdot)}^2_{L^2(\R)} + \frac{B^2\beta^2\eps}{2}\norm{\ptxx\ueb(t,\cdot)}^2_{L^2(\R)}.
\end{align*}
Therefore, \eqref{eq:L6} gives
\begin{equation}
\label{eq:L9}
\begin{split}
&\frac{d}{dt}\left(\frac{\beta^{\frac{1}{2}} + B\eps^2}{2}\norm{\px\ueb(t,\cdot)}^2_{L^2(\R)}+ \frac{A\beta\eps^2}{2} \norm{\pxx\ueb(t,\cdot)}^2_{L^2(\R)} \right)\\
&\qquad\quad + \frac{\beta^{\frac{5}{2}}}{2}\frac{d}{dt}\norm{\pxxx\ueb(t,\cdot)}^2_{L^2(\R)} + \beta\eps A\norm{\ptx\ueb(t,\cdot)}^2_{L^2(\R)}\\
&\qquad\quad +\frac{B\eps}{2}\norm{\pt\ueb(t,\cdot)}^2_{L^2(\R)} + \frac{A\beta^3\eps}{2}\norm{\ptxxx\ueb(t,\cdot)}^2_{L^2(\R)}\\
&\qquad\quad +\frac{B}{2}\beta^2\eps\left(1-B\right)\norm{\ptxx\ueb(t,\cdot)}^2_{L^2(\R)}+ \frac{\beta^{\frac{1}{2}}\eps}{2}\left(1-A\right)\norm{\pxx\ueb(t,\cdot)}^2_{L^2(\R)}\\
&\qquad \le  C_{0}\eps\norm{\ueb}^2_{L^2((0,T)\times\R)}\norm{\px\ueb(t,\cdot)}^2_{L^2(\R)}+\frac{\eps}{2}\norm{\px\ueb(t,\cdot)}^2_{L^2(\R)}\\
&\qquad\quad+\frac{2A^2\eps}{B}\norm{\ueb}^2_{L^2((0,T)\times\R)}\norm{\px\ueb(t,\cdot)}^2_{L^2(\R)}\\
&\qquad\quad + 2B\eps\norm{\ueb}^2_{L^2((0,T)\times\R)}\norm{\px\ueb(t,\cdot)}^2_{L^2(\R)}.
\end{split}
\end{equation}
Choosing $\displaystyle A=\frac{1}{2},\,B=\frac{1}{2}$, from \eqref{eq:L9}, we have
\begin{align*}
&\frac{d}{dt}\left(\frac{2\beta^{\frac{1}{2}} + \eps^2}{4}\norm{\px\ueb(t,\cdot)}^2_{L^2(\R)}+ \frac{\beta\eps^2}{4} \norm{\pxx\ueb(t,\cdot)}^2_{L^2(\R)} \right)\\
&\qquad\quad + \frac{\beta^{\frac{5}{2}}}{2}\frac{d}{dt}\norm{\pxxx\ueb(t,\cdot)}^2_{L^2(\R)} + \frac{\beta\eps}{2} \norm{\ptx\ueb(t,\cdot)}^2_{L^2(\R)}\\
&\qquad\quad  +\frac{\eps}{4}\norm{\pt\ueb(t,\cdot)}^2_{L^2(\R)} + \frac{\beta^3\eps}{4}\norm{\ptxxx\ueb(t,\cdot)}^2_{L^2(\R)}\\
&\qquad\quad  +\frac{\beta^2\eps}{8}\norm{\ptxx\ueb(t,\cdot)}^2_{L^2(\R)}+ \frac{\beta^{\frac{1}{2}}\eps}{4}\norm{\pxx\ueb(t,\cdot)}^2_{L^2(\R)}\\
&\qquad \le  C_{0}\eps\norm{\ueb}^2_{L^2((0,T)\times\R)}\norm{\px\ueb(t,\cdot)}^2_{L^2(\R)}+\frac{\eps}{2}\norm{\px\ueb(t,\cdot)}^2_{L^2(\R)}.
\end{align*}
\eqref{eq:u0eps-1}, \eqref{eq:l-2-u1}, and an integration on $(0,t)$ give
\begin{equation}
\label{eq:L10}
\begin{split}
&\frac{2\beta^{\frac{1}{2}} + \eps^2}{4}\norm{\px\ueb(t,\cdot)}^2_{L^2(\R)}+ \frac{\beta\eps^2}{4} \norm{\pxx\ueb(t,\cdot)}^2_{L^2(\R)}\\
&\qquad\quad +\frac{\beta^{\frac{5}{2}}}{2}\norm{\pxxx\ueb(t,\cdot)}^2_{L^2(\R)} + \frac{\beta\eps}{2} \int_{0}^{t}\norm{\ptx\ueb(s,\cdot)}^2_{L^2(\R)}ds\\
&\qquad\quad +\frac{\eps}{4}\int_{0}^{t}\norm{\pt\ueb(s,\cdot)}^2_{L^2(\R)}ds + \frac{\beta^3\eps}{4}\int_0^t\norm{\ptxxx\ueb(s,\cdot)}^2_{L^2(\R)}ds\\
&\qquad\quad+\frac{\beta^2\eps}{8}\int_{0}^{t}\norm{\ptxx\ueb(s,\cdot)}^2_{L^2(\R)}ds+ \frac{\beta^{\frac{1}{2}}\eps}{4}\int_{0}^{t}
\norm{\pxx\ueb(s,\cdot)}^2_{L^2(\R)}ds\\
&\qquad \le C_{0} + C_{0}\eps\norm{\ueb}^2_{L^2((0,T)\times\R)}\int_{0}^{t}\norm{\px\ueb(s,\cdot)}^2_{L^2(\R)}ds\\
&\qquad\quad+\frac{\eps}{2}\int_{0}^{t}\norm{\px\ueb(s,\cdot)}^2_{L^2(\R)}ds \le C_{0}\left(1+\norm{\ueb}^2_{L^2((0,T)\times\R)}\right).
\end{split}
\end{equation}
We prove \eqref{eq:u-infty-3}. Due to \eqref{eq:l-2-u1}, \eqref{eq:L10}, and the H\"older inequality,
\begin{align*}
\ueb^2(t,x) =& 2\int_{-\infty}^{x}\ueb\px\ueb dx \le 2\int_{\R}\vert\ueb\px\ueb\vert dx \\
\le & \norm{\ueb(t,\cdot)}_{L^2(\R)}\norm{\px\ueb(t,\cdot)}_{L^2(\R)}\\
\le& \frac{C_{0}}{\beta^{\frac{1}{4}}}\sqrt{\left(1+\norm{\ueb}^2_{L^2((0,T)\times\R)}\right)},
\end{align*}
that is
\begin{equation*}
\norm{\ueb}^4_{L^2((0,T)\times\R)}\frac{C_0}{\beta^{\frac{1}{2}}}\left(1+\norm{\ueb}^2_{L^2((0,T)\times\R)}\right).
\end{equation*}
Arguing as \cite[Lemma $2.2$]{Cd6}, we have \eqref{eq:u-infty-3}.

It follows from \eqref{eq:u-infty-3} and \eqref{eq:L10} that
\begin{align*}
&\frac{2\beta^{\frac{1}{2}} + \eps^2}{4}\norm{\px\ueb(t,\cdot)}^2_{L^2(\R)}+ \frac{\beta\eps^2}{4} \norm{\pxx\ueb(t,\cdot)}^2_{L^2(\R)}\\
&\qquad\quad +\frac{\beta^{\frac{5}{2}}}{2}\norm{\pxxx\ueb(t,\cdot)}^2_{L^2(\R)} + \frac{\beta\eps}{2} \int_{0}^{t}\norm{\ptx\ueb(s,\cdot)}^2_{L^2(\R)}ds\\
&\qquad\quad +\frac{\eps}{4}\int_{0}^{t}\norm{\pt\ueb(s,\cdot)}^2_{L^2(\R)}ds + \frac{\beta^3\eps}{4}\int_{\R}\norm{\ptxxx\ueb(s,\cdot)}^2_{L^2(\R)}ds\\
&\qquad\quad+\frac{\beta^2\eps}{8}\int_{0}^{t}\norm{\ptxx\ueb(s,\cdot)}^2_{L^2(\R)}ds+ \frac{\beta^{\frac{1}{2}}\eps}{4}\int_{0}^{t}\norm{\pxx\ueb(t,\cdot)}^2_{L^2(\R)}ds \le C_{0}\beta^{-\frac{1}{2}},
\end{align*}
that is,
\begin{align*}
&\frac{2\beta + \beta^{\frac{1}{2}}\eps^2}{4}\norm{\px\ueb(t,\cdot)}^2_{L^2(\R)}+ \frac{\beta^{\frac{3}{2}}\eps^2}{4} \norm{\pxx\ueb(t,\cdot)}^2_{L^2(\R)}\\
&\qquad\quad +\frac{\beta^3}{2}\norm{\pxxx\ueb(t,\cdot)}^2_{L^2(\R)} + \frac{\beta^{\frac{3}{2}}\eps}{2} \int_{0}^{t}\norm{\ptx\ueb(s,\cdot)}^2_{L^2(\R)}ds\\
&\qquad\quad +\frac{\beta^{\frac{1}{2}}\eps}{4}\int_{0}^{t}\norm{\pt\ueb(s,\cdot)}^2_{L^2(\R)}ds + \frac{\beta^{\frac{7}{2}}\eps}{4}\int_{\R}\norm{\ptxxx\ueb(s,\cdot)}^2_{L^2(\R)}ds\\
&\qquad\quad+\frac{\beta^{\frac{5}{2}}\eps}{8}\int_{0}^{t}\norm{\ptxx\ueb(s,\cdot)}^2_{L^2(\R)}ds+ \frac{\beta\eps}{4}\int_{0}^{t}\norm{\pxx\ueb(t,\cdot)}^2_{L^2(\R)}ds \le C_{0}.
\end{align*}
Hence,
\begin{align*}
\beta^{\frac{1}{2}}\norm{\px\ueb(t,\cdot)}_{L^2(\R)}\le &C_{0},\\
\beta^{\frac{1}{4}}\eps\norm{\px\ueb(t,\cdot)}_{L^2(\R)}\le &C_{0},\\
\beta^{\frac{3}{4}}\eps\norm{\pxx\ueb(t,\cdot)}_{L^2(\R)}\le &C_{0},\\
\beta^{\frac{3}{2}}\norm{\pxxx\ueb(t,\cdot)}_{L^2(\R)}\le &C_{0},\\
\beta^{\frac{3}{2}}\eps\int_{0}^{t}\norm{\ptx\ueb(s,\cdot)}^2_{L^2(\R)}ds\le&C_{0},\\
\beta^{\frac{1}{2}}\eps\int_{0}^{t}\norm{\pt\ueb(s,\cdot)}^2_{L^2(\R)}ds\le&C_{0},\\
\beta^{\frac{7}{2}}\eps\int_{\R}\norm{\ptxxx\ueb(s,\cdot)}^2_{L^2(\R)}ds\le&C_{0},\\
\beta^{\frac{5}{2}}\eps\int_{0}^{t}\norm{\ptxx\ueb(s,\cdot)}^2_{L^2(\R)}ds\le&C_{0},\\
\beta\eps\int_{0}^{t}\norm{\pxx\ueb(t,\cdot)}^2_{L^2(\R)}ds\le&C_{0},
\end{align*}
for every $0<t<T$.
\end{proof}
To prove Theorem \ref{th:main-1}. The following technical lemma is needed  \cite{Murat:Hneg}.
\begin{lemma}
\label{lm:1}
Let $\Omega$ be a bounded open subset of $
\R^2$. Suppose that the sequence $\{\mathcal
L_{n}\}_{n\in\mathbb{N}}$ of distributions is bounded in
$W^{-1,\infty}(\Omega)$. Suppose also that
\begin{equation*}
\mathcal L_{n}=\mathcal L_{1,n}+\mathcal L_{2,n},
\end{equation*}
where $\{\mathcal L_{1,n}\}_{n\in\mathbb{N}}$ lies in a
compact subset of $H^{-1}_{loc}(\Omega)$ and
$\{\mathcal L_{2,n}\}_{n\in\mathbb{N}}$ lies in a
bounded subset of $\mathcal{M}_{loc}(\Omega)$. Then $\{\mathcal
L_{n}\}_{n\in\mathbb{N}}$ lies in a compact subset of $H^{-1}_{loc}(\Omega)$.
\end{lemma}
Moreover, we consider the following definition.
\begin{definition}
A pair of functions $(\eta, q)$ is called an  entropy--entropy flux pair if\\
 $\eta :\R\to\R$ is a $C^2$ function and $q :\R\to\R$ is defined by
\begin{equation*}
q(u)=2\int_{0}^{u} \xi\eta'(\xi) d\xi.
\end{equation*}
An entropy-entropy flux pair $(\eta,\, q)$ is called  convex/compactly supported if, in addition, $\eta$ is convex/compactly supported.
\end{definition}
We begin by proving the following result
\begin{lemma}\label{lm:259}
Assume that \eqref{eq:uo-l2}, \eqref{eq:u0eps-1} and \eqref{eq:beta-eps-2} hold. Then for any compactly
supported entropy--entropy flux pair $(\eta, \,q)$, there exist two sequences $\{\eps_{n}\}_{n\in\N},\,\{\beta_{n}\}_{n\in\N}$, with $\eps_n,\,\beta_n\to0$, and a limit function
\begin{equation*}
u\in L^{\infty}((0,T);L^2(\R)),
\end{equation*}
such that
\begin{align}
\label{eq:con-u-1}
&u_{\eps_{n},\,\beta_{n}}\to u \quad \textrm{in $L^p_{loc}((0,T)\times\R)$, for each $1\le p<2$},\\
\label{eq:u-dist12}
&u \quad \textrm{is a distributional solution of \eqref{eq:BU}}.
\end{align}
\end{lemma}
\begin{proof}
Let us consider a compactly supported entropy--entropy flux pair $(\eta, q)$. Multiplying \eqref{eq:Ro-eps-beta} by $\eta'(\ueb)$, we have
\begin{align*}
\pt\eta(\ueb) + \px q(\ueb) =&\eps \eta'(\ueb) \pxx\ueb -\beta\pxxx\ueb -\beta^2\eta'(\ueb)\ptxxxx\ueb \\
=& I_{1,\,\eps,\,\beta}+I_{2,\,\eps,\,\beta}+ I_{3,\,\eps,\,\beta} + I_{4,\,\eps,\,\beta}+I_{5,\,\eps,\,\beta} + I_{6,\,\eps,\,\beta},
\end{align*}
where
\begin{equation}
\begin{split}
\label{eq:12000}
I_{1,\,\eps,\,\beta}&=\px(\eps\eta'(\ueb)\px\ueb),\\
I_{2,\,\eps,\,\beta}&= -\eps\eta''(\ueb)(\px\ueb)^2,\\
I_{3,\,\eps,\,\beta}&= -\px(\beta\eta'(\ueb)\pxx\ueb),\\
I_{4,\,\eps,\,\beta}&= \beta\eta''(\ueb)\px\ueb\pxx\ueb,\\
I_{5,\,\eps,\,\beta}&= -\px(\beta^2\eta'(\ueb)\ptxxx\ueb),\\
I_{6,\,\eps,\,\beta}&= \beta^2\eta''(\ueb)\px\ueb\ptxxx\ueb.
\end{split}
\end{equation}
Fix $T>0$. Arguing in \cite[Lemma $3.2$]{Cd2}, we have that $I_{1,\,\eps,\,\beta}\to0$ in $H^{-1}((0,T) \times\R)$, and $\{I_{2,\,\eps,\,\beta}\}_{\eps,\beta >0}$ is bounded in $L^1((0,T)\times\R)$. Arguing in \cite[Theorem $B.1$]{Cd6},  $I_{3,\,\eps,\,\beta}\to0$ in $H^{-1}((0,T) \times\R)$, and $I_{4,\,\eps,\,\beta}\to0$ in $L^{1}((0,T) \times\R)$, while arguing in \cite[Lemma $2.4$]{Cd6}, $I_{5,\,\eps,\,\beta}\to0$ in $H^{-1}((0,T) \times\R)$,
and $\{I_{6,\,\eps,\,\beta}\}_{\eps,\beta >0}$ is bounded in $L^1((0,T)\times\R)$.

Therefore, \eqref{eq:con-u-1} follows from Lemmas \ref{lm:38}, \ref{lm:1} and the $L^p$ compensated compactness of \cite{SC}.\\
Arguing in \cite[Theorem $2.1$]{Cd5}, we have \eqref{eq:u-dist12}.
\end{proof}
Following \cite{LN}, we prove the following result
\begin{lemma}\label{lm:452}
Assume that \eqref{eq:uo-l2}, \eqref{eq:u0eps-1} and \eqref{eq:beta-eps-4} hold. Then  for any compactly
supported entropy--entropy flux pair $(\eta, \,q)$, there exist two sequences $\{\eps_{n}\}_{n\in\N},\,\{\beta_{n}\}_{n\in\N}$, with $\eps_n,\,\beta_n\to0$, and a limit function
\begin{equation*}
u\in L^{\infty}((0,T);L^2(\R)),
\end{equation*}
such that
\eqref{eq:con-u-1} holds and
\begin{equation}
\label{eq:u-entro-sol-12}
u \quad \textrm{is the unique entropy solution of \eqref{eq:BU}}.
\end{equation}
\end{lemma}
\begin{proof}
Let us consider a compactly supported entropy-entropy flux pair $(\eta,\,q)$. Multiplying \eqref{eq:Ro-eps-beta} by $\eta'(\ueb)$, we have
\begin{align*}
\pt\eta(\ueb) + \px q(\ueb) =&\eps \eta'(\ueb) \pxx\ueb-\beta\pxxx\ueb -\beta^2\eta'(\ueb)\ptxxxx\ueb \\
=& I_{1,\,\eps,\,\beta}+I_{2,\,\eps,\,\beta}+ I_{3,\,\eps,\,\beta} + I_{4,\,\eps,\,\beta}+ I_{5,\,\eps,\,\beta} + I_{6,\,\eps,\,\beta}
\end{align*}
where $I_{1,\,\eps,\,\beta},\,I_{2,\,\eps,\,\beta},\, I_{3,\,\eps,\,\beta},\, I_{4,\,\eps,\,\beta},\,I_{5,\,\eps,\,\beta},\,I_{6,\,\eps,\,\beta}$ are defined in \eqref{eq:12000}.

As in Lemma \ref{lm:259}, we obtain that $I_{1,\,\eps,\,\beta}\to 0$ in $H^{-1}((0,T)\times\R)$, $\{I_{2,\,\eps,\,\beta}\}_{\eps,\beta>0}$ is bounded in $L^1((0,T)\times\R)$,  $I_{3,\,\eps,\,\beta}\to 0$ in $H^{-1}((0,T)\times\R)$, $I_{4,\,\eps,\,\beta}\to 0$ in $L^1((0,T)\times\R)$, $I_{5,\,\eps,\,\beta}\to 0$ in $H^{-1}((0,T)\times\R)$, while arguing in \cite[Lemma $2.4$]{Cd6}, $I_{6,\,\eps,\,\beta}\to 0$ in $L^1((0,T)\times\R)$

Arguing in \cite[Theorem $2.1$]{Cd5}, we have \eqref{eq:u-entro-sol-12}.
\end{proof}
\begin{proof}[Proof of Theorem \ref{th:main-1}]
Theorem \ref{th:main-1} follows from Lemmas \ref{lm:259} and \ref{lm:452}.
\end{proof}

\section{The Rosenau-KdV-equation. $u_0\in L^2(\R)\cap L^4(\R)$.}\label{sec:D1}
In this section,  we consider \eqref{eq:RKV33}, and  assume \eqref{eq:uo-l4} on the initial datum.
We consider the approximation \eqref{eq:Ro-eps-beta}, where $u_{\eps,\beta,0}$ is a $C^\infty$ approximation of $u_{0}$ such that
\begin{equation}
\begin{split}
\label{eq:u0eps-14}
&u_{\eps,\,\beta,\,0} \to u_{0} \quad  \textrm{in $L^{p}_{loc}(\R)$, $1\le p < 2$, as $\eps,\,\beta \to 0$,}\\
&\norm{u_{\eps,\beta, 0}}^4_{L^4(\R)}+\norm{u_{\eps,\beta, 0}}^2_{L^2(\R)}+\left(\beta^{\frac{1}{2}}+ \eps^2\right) \norm{\px u_{\eps,\beta,0}}^2_{L^2(\R)}\le C_0,\quad \eps,\beta >0,\\
&\left(\beta^2+\beta\eps^2\right) \norm{\pxx u_{\eps,\beta,0}}^2_{L^2(\R)} +\left(\beta^{\frac{5}{2}}+\beta^2\eps^2\right)\norm{\pxxx u_{\eps,\beta,0}}^2_{L^2(\R)}\le C_0,\quad \eps,\beta >0,\\
&\beta^4\norm{\pxxxx u_{\eps,\beta,0}}^2_{L^2(\R)}\le C_0,\quad \eps,\beta >0,
\end{split}
\end{equation}
and $C_0$ is a constant independent on $\eps$ and $\beta$.

The main result of this section is the following theorem.
\begin{theorem}
\label{th:main-13}
Assume that \eqref{eq:uo-l4} and  \eqref{eq:u0eps-14} hold. Fix $T>0$,
if \eqref{eq:beta-eps-2} holds, there exist two sequences $\{\eps_{n}\}_{n\in\N}$, $\{\beta_{n}\}_{n\in\N}$, with $\eps_n, \beta_n \to 0$, and a limit function
\begin{equation*}
u\in L^{\infty}((0,T); L^2(\R)\cap L^4(\R)),
\end{equation*}
such that
\begin{itemize}
\item[$i)$] $u_{\eps_n, \beta_n}\to u$  strongly in $L^{p}_{loc}(\R^{+}\times\R)$, for each $1\le p <4$,
\item[$ii)$] $u$ is the unique entropy solution of \eqref{eq:BU}.
\end{itemize}
\end{theorem}
Let us prove some a priori estimates on $\ueb$, denoting with $C_0$ the constants which depend only on the initial data.
\begin{lemma}\label{lm:562}
Fix $T>0$. Assume \eqref{eq:beta-eps-2} holds. There exists $C_0>0$, independent on $\eps,\,\beta$ such that \eqref{eq:u-infty-3} holds.
In particular, we have
\begin{equation}
\label{eq:Z45}
\begin{split}
\beta\norm{\px\ueb(t,\cdot)}^2_{L^2(\R)}&+ \beta^3\norm{\pxxx\ueb(t,\cdot)}^2_{L^2(\R)}\\ &+\frac{3\beta\eps}{2}\int_{0}^{t}\norm{\pxx\ueb(s,\cdot)}^2_{L^2(\R)}ds\le C_0,
\end{split}
\end{equation}
for every $0<t<T$. Moreover,
\begin{equation}
\label{eq:Z46}
\norm{\px\ueb}_{L^{\infty}((0,T)\times\R)}\le C_0\beta^{-\frac{3}{4}}.
\end{equation}
\end{lemma}
\begin{remark}
Observe that the proof of Lemma \ref{lm:562} is simpler than the one of Lemma \ref{lm:50}. Indeed, we only need to prove \eqref{eq:u-infty-3}.
\end{remark}
\begin{proof}[Proof of Lemma \ref{lm:562}.]
Let $0<t<T$. Multiplying \eqref{eq:Ro-eps-beta} by $-\beta^{\frac{1}{2}}\pxx\ueb$, we have
\begin{equation}
\label{eq:K23}
\begin{split}
-\beta^{\frac{1}{2}}\pxx\ueb\pt\ueb &- 2\beta^{\frac{1}{2}}\ueb\px\ueb\pxx\ueb\\
&+\beta^{\frac{3}{2}}\pxx\ueb\pxxx\ueb-\beta^{\frac{5}{2}}\ptxxxx\ueb\pxx\ueb=  -\beta^{\frac{1}{2}}\eps(\pxx\ueb)^2.
\end{split}
\end{equation}
We note that
\begin{equation*}
\beta^{\frac{3}{2}}\int_{\R}\pxx\ueb\pxxx\ueb dx =0.
\end{equation*}
Therefore, arguing as \cite[Lemma $3.1$]{Cd6}, we have \eqref{eq:u-infty-3}, \eqref{eq:Z45} and \eqref{eq:Z46}.
\end{proof}
Following \cite[Lemma $2.2$]{Cd}, or \cite[Lemma $4.2$]{CK}, we prove the following result.
\begin{lemma}\label{lm:t3}
Fix $T>0$. Assume \eqref{eq:beta-eps-2} holds. Then:
\begin{itemize}
\item[$i)$] the family  $\{\ueb\}_{\eps,\,\beta}$ is bounded in $L^{\infty}((0,T);L^{4}(\R))$;
\item[$ii)$] the families $\{\eps\px\ueb\}_{\eps,\,\beta},\,\{\beta^{\frac{1}{2}}\eps\pxx\ueb\}_{\eps,\beta},\, \{\beta\pxx\ueb\}_{\eps,\beta},$\\
 $\{\beta\eps\pxxx\ueb\}_{\eps,\beta},\, \{\beta\pxxxx\ueb\}_{\eps,\beta} $are bounded in $L^{\infty}((0,T);L^{2}(\R))$;
\item[$iii)$] the families $\{\beta^{\frac{1}{2}}\eps^{\frac{1}{2}}\ptx\ueb\}_{\eps,\,\beta},\,\{\eps^{\frac{1}{2}}\pt\ueb\}_{\eps,\,\beta},\, \{\beta^{\frac{3}{2}}\eps^{\frac{1}{2}}\ptxxx\ueb\}_{\eps,\,\beta},$\\
$\{\beta\eps^{\frac{1}{2}}\ptxx\ueb\}_{\eps,\,\beta},\,\{\eps^{\frac{1}{2}}\ueb\px\ueb\}_{\eps,\,\beta}
 \{\eps^{\frac{3}{2}}\pxx\ueb\}_{\eps,\,\beta},\, \{\beta\eps^{\frac{1}{2}}\pxxx\ueb\}_{\eps,\,\beta},\,$ are bounded in $L^{2}((0,T)\times\R)$;
\end{itemize}
\end{lemma}
\begin{proof}
Let $0<t<T$. Let $A,\,B,\,C,\,E$ be some positive constants which will be specified later. Multiplying \eqref{eq:Ro-eps-beta} by
\begin{equation*}
\ueb^3 -A\eps^2\pxx\ueb -B\beta\eps\ptxx\ueb +C\eps\pt\ueb +E\beta^2\pxxxx\ueb,
\end{equation*}
we have
\begin{equation}
\label{eq:P1}
\begin{split}
&\left(\ueb^3 -A\eps^2\pxx\ueb -B\beta\eps\ptxx\ueb\right)\pt\ueb\\
&\qquad\quad +\left(C\eps\pt\ueb +E\beta^2\pxxxx\ueb\right)\pt\ueb\\
&\qquad\quad+2\left(\ueb^3 -A\eps^2\pxx\ueb -B\beta\eps\ptxx\ueb\right)\ueb\px\ueb\\
&\qquad\quad +2\left(C\eps\pt\ueb +E\beta^2\pxxxx\ueb\right)\ueb\px\ueb\\
&\qquad\quad +\beta\left(\ueb^3 -A\eps^2\pxx\ueb -B\beta\eps\ptxx\ueb\right)\pxxx\ueb\\
&\qquad\quad +\beta\left(C\eps\pt\ueb +E\beta^2\pxxxx\ueb\right)\pxxx\ueb\\
&\qquad\quad +\beta^2 \left(\ueb^3 -A\eps^2\pxx\ueb -B\beta\eps\ptxx\ueb\right)\ptxxxx\ueb\\
&\qquad\quad +\beta^2\left(C\eps\pt\ueb +E\beta^2\pxxxx\ueb\right)\ptxxxx\ueb\\
&\qquad= \eps\left(\ueb^3 -A\eps^2\pxx\ueb -B\beta\eps\ptxx\ueb\right)\pxx\ueb\\
&\qquad\quad +\eps \left(C\eps\pt\ueb +E\beta^2\pxxxx\ueb\right)\pxx\ueb.
\end{split}
\end{equation}
Since
\begin{align*}
\int_{\R}&\left(\ueb^3 -A\eps^2\pxx\ueb -B\beta\eps\ptxx\ueb\right)\pt\ueb dx\\
=&\frac{1}{4}\frac{d}{dt}\norm{\ueb(t,\cdot)}^4_{L^4(\R)} + \frac{A\eps^2}{2}\frac{d}{dt}\norm{\px\ueb(t,\cdot)}^2_{L^2(\R)} \\
& +B\beta\eps\norm{\ptx\ueb(t,\cdot)}^2_{L^2(\R)},\\
\int_{\R}&\left(C\eps\pt\ueb +E\beta^2\pxxxx\ueb\right)\pt\ueb dx\\
=& C\eps\norm{\pt\ueb(t,\cdot)}^2_{L^2(\R)}+ \frac{E\beta^2}{2}\frac{d}{dt}\norm{\pxx\ueb(t,\cdot)}^2_{L^2(\R)},\\
2\int_{\R}&\left(\ueb^3 -A\eps^2\pxx\ueb -B\beta\eps\ptxx\ueb\right)\ueb\px\ueb dx \\
=& -2A\eps^2\int_{\R}\ueb\px\ueb\pxx\ueb dx -2B\beta\eps\int_{\R}\ueb\px\ueb\ptxx\ueb dx,\\
2\int_{\R}&\left(C\eps\pt\ueb +E\beta^2\pxxxx\ueb\right)\ueb\px\ueb dx\\
=& 2C\int_{\R}\ueb\px\ueb\pt\ueb dx -2E\beta^2\int_{\R}(\px\ueb)^2\pxxx\ueb dx\\
&-2E\beta^2\int_{\R}\ueb\pxx\ueb\pxxx\ueb dx,\\
-2E\beta^2\int_{\R}&(\px\ueb)^2\pxxx\ueb dx-2E\beta^2\int_{\R}\ueb\pxx\ueb\pxxx\ueb dx\\
=&5E\beta^2\int_{\R}(\pxx\ueb)^2\px\ueb dx=-\frac{5E\beta^2}{2}\int_{\R}(\px\ueb)^2\pxxx\ueb dx,\\
2\int_{\R}&\left(C\eps\pt\ueb +E\beta^2\pxxxx\ueb\right)\ueb\px\ueb dx\\
=& 2C\eps\int_{\R}\ueb\px\ueb\pt\ueb dx-\frac{5E\beta^2}{2}\int_{\R}(\px\ueb)^2\pxxx\ueb dx,\\
\beta\int_{\R}&\left(\ueb^3 -A\eps^2\pxx\ueb -B\beta\eps\ptxx\ueb\right)\pxxx\ueb dx\\
=& -3\beta\int_{\R}\ueb^2\px\ueb\pxx\ueb dx - B\beta^2\eps\int_{\R}\ptxx\ueb\pxxx\ueb dx,\\
\beta\int_{\R}&\left(C\eps\pt\ueb +E\beta^2\pxxxx\ueb\right)\pxxx\ueb=C\beta\eps\int_{\R}\ptxx\ueb\px\ueb dx,\\
\beta^2\int_{\R}&\left(\ueb^3 -A\eps^2\pxx\ueb -B\beta\eps\ptxx\ueb\right)\ptxxxx\ueb dx\\
=& -3\beta^2 \int_{\R}\ueb^2\px\ueb \ptxxx\ueb dx +\frac{A\beta^2\eps^2}{2}\frac{d}{dt}\norm{\pxxx\ueb(t,\cdot)}^2_{L^2(\R)}\\
&+B\beta^3\eps\norm{\ptxxx\ueb(t,\cdot)}^2_{L^2(\R)}, \\
\beta^2\int_{\R}&\left(C\eps\pt\ueb +E\beta^2\pxxxx\ueb\right)\ptxxxx\ueb dx \\
=& C\beta^2\eps\norm{\ptxx\ueb(t,\cdot)}^2_{L^2(\R)} +\frac{E\beta^4}{2}\frac{d}{dt}\norm{\pxxxx\ueb(t,\cdot)}^2_{L^2(\R)},\\
\eps\int_{\R}&\left(\ueb^3 -A\eps^2\pxx\ueb -B\beta\eps\ptxx\ueb\right)\pxx\ueb dx\\
=& -3\eps\norm{\ueb(t,\cdot)\px\ueb(t,\cdot)}^2_{L^2(\R)} -A \eps^3\norm{\pxx\ueb(t,\cdot)}^2_{L^2(\R)}\\
& -\frac{B\beta\eps^2}{2}\frac{d}{dt}\norm{\pxx\ueb(t,\cdot)}^2_{L^2(\R)},\\
\eps\int_{\R}& \left(C\eps\pt\ueb +E\beta^2\pxxxx\ueb\right)\pxx\ueb dx\\
=& -\frac{C\eps^2}{2}\frac{d}{dt}\norm{\px\ueb(t,\cdot)}^2_{L^2(\R)} -E\beta^2\eps\norm{\pxxx\ueb(t,\cdot)}^2_{L^2(\R)}.
\end{align*}
 an integration on $\R$ of \eqref{eq:P1} gives
\begin{equation}
\label{eq:P9}
\begin{split}
&\frac{d}{dt}\left( \frac{1}{4}\norm{\ueb(t,\cdot)}^4_{L^4(\R)}+ \frac{\left(A+C\right)\eps^2}{2}\norm{\px\ueb(t,\cdot)}^2_{L^2(\R)}  \right)\\
&\qquad\quad+\frac{d}{dt}\left(\frac{A\beta^2\eps^2}{2}\norm{\pxxx\ueb(t,\cdot)}^2_{L^2(\R)}+ \frac{E\beta^4}{2}\norm{\pxxxx\ueb(t,\cdot)}^2_{L^2(\R)}\right)\\
&\qquad\quad +\frac{B\beta\eps^2+E\beta^2}{2}\frac{d}{dt}\norm{\pxx\ueb(t,\cdot)}^2_{L^2(\R)} +B\beta\eps\norm{\ptx\ueb(t,\cdot)}^2_{L^2(\R)}\\
&\qquad\quad +C\eps\norm{\pt\ueb(t,\cdot)}^2_{L^2(\R)}+B\beta^3\eps\norm{\ptxxx\ueb(t,\cdot)}^2_{L^2(\R)}\\
&\qquad\quad+C\beta^2\eps\norm{\ptxx\ueb(t,\cdot)}^2_{L^2(\R)} +3\eps\norm{\ueb(t,\cdot)\px\ueb(t,\cdot)}^2_{L^2(\R)}\\
&\qquad\quad +A \eps^3\norm{\pxx\ueb(t,\cdot)}^2_{L^2(\R)}+E\beta^2\eps\norm{\pxxx\ueb(t,\cdot)}^2_{L^2(\R)}\\
&\qquad= 2A\eps^2\int_{\R}\ueb\px\ueb\pxx\ueb dx +2B\beta\eps\int_{\R}\ueb\px\ueb\ptxx\ueb dx\\
&\qquad\quad +2C\eps\int_{\R}\ueb\px\ueb\pt\ueb dx+\frac{5E\beta^2}{2}\int_{\R}(\px\ueb)^2\pxxx\ueb dx\\
&\qquad\quad +3\beta\int_{\R}\ueb^2\px\ueb\pxx\ueb dx + B\beta^2\eps\int_{\R}\ptxx\ueb\pxxx\ueb dx\\
&\qquad\quad -C\beta\eps\int_{\R}\ptxx\ueb\px\ueb dx +3\beta^2 \int_{\R}\ueb^2\px\ueb \ptxxx\ueb dx.
\end{split}
\end{equation}
Due to the Young inequality,
\begin{align*}
&2A\eps^2\int_{\R}\vert\ueb\px\ueb\vert\vert\pxx\ueb\vert dx = \int_{\R}\left\vert \eps^{\frac{1}{2}}\ueb\px\ueb\right\vert\left\vert 2A \eps^{\frac{3}{2}}\pxx\ueb\right\vert dx\\
&\qquad\le \frac{\eps}{2}\norm{\ueb(t,\cdot)\px\ueb(t,\cdot)}^2_{L^2(\R)} + 2A^2\eps^3\norm{\pxx\ueb(t,\cdot)}^2_{L^2(\R)},\\
&2B\beta\eps\int_{\R}\ueb\px\ueb\ptxx\ueb dx = \eps\int_{\R}\left\vert\ueb\px\ueb\right\vert\left\vert 2B\beta\ptxx\ueb\right\vert dx\\
&\qquad \frac{\eps}{2}\norm{\ueb(t,\cdot)\px\ueb(t,\cdot)}^2_{L^2(\R)} + 4B^2\beta^2\eps\norm{\ptxx\ueb(t,\cdot)}^2_{L^2(\R)},\\
&2C\eps\int_{\R}\vert\ueb\px\ueb\vert\pt\ueb dx= \eps\int_{\R}\left\vert\ueb\px\ueb\right\vert\left\vert 2C\pt\ueb\right\vert dx\\
&\qquad \le \frac{\eps}{2}\norm{\ueb(t,\cdot)\px\ueb(t,\cdot)}^2_{L^2(\R)} +2C^2\eps\norm{\pt\ueb(t,\cdot)}^2_{L^2(\R)},\\
&B\beta^2\eps\int_{\R}\vert\ptxx\ueb\vert\vert\pxxx\ueb\vert dx= \beta^2\eps\int_{\R}\left\vert 2B\ptxx\ueb \right\vert\left\vert\frac{\pxxx\ueb}{2}\right\vert dx\\
&\qquad \le 4B^2\beta^2\eps\norm{\ptxx\ueb(t,\cdot)}^2_{L^2(\R)} +\frac{\beta^2\eps}{2}\norm{\pxxx\ueb(t,\cdot)}^2_{L^2(\R)},\\
&C\beta\eps\int_{\R}\vert\ptxx\ueb\vert\vert\px\ueb\vert dx = C\eps\int_{\R}\left\vert\beta\ptxx\ueb\right\vert\left\vert\px\ueb\right\vert dx\\
&\qquad\le \frac{C\beta^2\eps}{2}\norm{\ptxx\ueb(t,\cdot)}^2_{L^2(\R)}+\frac{C\eps}{2}\norm{\px\ueb(t,\cdot)}^2_{L^2(\R)}.
\end{align*}
Therefore, from \eqref{eq:P9}, we have
\begin{equation}
\label{eq:P11}
\begin{split}
&\frac{d}{dt}\left( \frac{1}{4}\norm{\ueb(t,\cdot)}^4_{L^4(\R)}+ \frac{\left(A+C\right)\eps^2}{2}\norm{\px\ueb(t,\cdot)}^2_{L^2(\R)}  \right)\\
&\qquad\quad+\frac{d}{dt}\left(\frac{A\beta^2\eps^2}{2}\norm{\pxxx\ueb(t,\cdot)}^2_{L^2(\R)}+ \frac{E\beta^4}{2}\norm{\pxxxx\ueb(t,\cdot)}^2_{L^2(\R)}\right)\\
&\qquad\quad +\frac{B\beta\eps^2+E\beta^2}{2}\frac{d}{dt}\norm{\pxx\ueb(t,\cdot)}^2_{L^2(\R)} +B\beta\eps\norm{\ptx\ueb(t,\cdot)}^2_{L^2(\R)}\\
&\qquad\quad +\left(1-2C\right)C\eps\norm{\pt\ueb(t,\cdot)}^2_{L^2(\R)} +B\beta^3\eps\norm{\ptxxx\ueb(t,\cdot)}^2_{L^2(\R)}\\
&\qquad\quad +\left(\frac{C}{2} -8B^2\right)\beta^2\eps\norm{\ptxx\ueb(t,\cdot)}^2_{L^2(\R)}+\frac{3\eps}{2}\norm{\ueb(t,\cdot)\px\ueb(t,\cdot)}^2_{L^2(\R)}\\
&\qquad\quad  +\left(A -2A^2\right)\eps^3\norm{\pxx\ueb(t,\cdot)}^2_{L^2(\R)}+\left(E-\frac{1}{2}\right)\beta^2\eps\norm{\pxxx\ueb(t,\cdot)}^2_{L^2(\R)}\\
&\qquad \le \frac{5E\beta^2}{2}\int_{\R}(\px\ueb)^2\vert\pxxx\ueb\vert dx + 3\beta\int_{\R}\ueb^2\vert\px\ueb\vert\vert\pxx\ueb\vert dx\\
&\qquad\quad \le 3\beta^2 \int_{\R}\ueb^2\vert\px\ueb\vert \vert\ptxxx\ueb\vert dx + \frac{C\eps}{2}\norm{\px\ueb(t,\cdot)}^2_{L^2(\R)}.
\end{split}
\end{equation}
From \eqref{eq:beta-eps-2}, we get
\begin{equation}
\label{eq:P12}
\beta\le D^2\eps^4,
\end{equation}
where $D$ is a positive constant that which will be specified later. It follows from \eqref{eq:Z46}, \eqref{eq:P12} and, the Young inequality that
\begin{align*}
&\frac{5E\beta^2}{2}\int_{\R}(\px\ueb)^2\vert\pxxx\ueb\vert dx=E\beta^2\int_{\R}\frac{5}{2\eps^{\frac{1}{2}}}(\px\ueb)^2\left\vert\eps^{\frac{1}{2}}\pxxx\ueb\right\vert dx\\
&\qquad\le \frac{25E\beta^2}{8}{\eps}\int_{\R}(\px\ueb)^4 dx  + \frac{E\beta^2\eps}{2}\norm{\pxxx\ueb(t,\cdot)}^2_{L^2(\R)}\\
&\qquad\le \frac{25E\beta^2}{8\eps}\norm{\px\ueb}^2_{L^{\infty}((0,T)\times\R)}\norm{\px\ueb(t,\cdot)}^2_{L^2(\R)}\\
&\qquad\quad + \frac{E\beta^2\eps}{2}\norm{\pxxx\ueb(t,\cdot)}^2_{L^2(\R)}\\
&\qquad\le \frac{C_{0}\beta^{\frac{1}{2}}}{\eps}\norm{\px\ueb(t,\cdot)}^2_{L^2(\R)} + \frac{E\beta^2\eps}{2}\norm{\pxxx\ueb(t,\cdot)}^2_{L^2(\R)}\\
&\qquad\le C_0D\norm{\px\ueb(t,\cdot)}^2_{L^2(\R)}+ \frac{E\beta^2\eps}{2}\norm{\pxxx\ueb(t,\cdot)}^2_{L^2(\R)},\\
&3\beta\int_{\R}\ueb^2\vert\px\ueb\vert\vert\pxx\ueb\vert dx \le 3\beta\norm{\ueb}^2_{L^{\infty}((0,T)\times\R)}\int_{\R}\vert\px\ueb\vert\vert\pxx\ueb\vert dx\\
&\qquad \le 3C_{0}D\eps^2\int_{\R}\vert\px\ueb\vert\vert\pxx\ueb\vert dx = 3\int_{\R}\left\vert\eps^{\frac{1}{2}}\px\ueb\right\vert \left\vert C_{0}D\eps^{\frac{3}{2}}\pxx\ueb\right\vert dx\\
&\qquad \le \frac{3\eps}{2}\norm{\px\ueb(t,\cdot)}^2_{L^2(\R)} + C^2_{0}D^2\eps^3\norm{\pxx\ueb(t,\cdot)}^2_{L^2(\R)},\\
&3\beta^2 \int_{\R}\ueb^2\vert\px\ueb\vert \vert\ptxxx\ueb\vert dx= \int_{\R}\left\vert\frac{3\beta^{\frac{1}{2}}\ueb^2\px\ueb}{\sqrt{B}\eps^{\frac{1}{2}}}\right\vert\left\vert\sqrt{B}\beta^{\frac{3}{2}}\eps^{\frac{1}{2}}\ptxxx\ueb\right\vert dx\\
&\qquad\le \frac{3\beta}{2B\eps}\int_{\R}\ueb^4(\px\ueb)^2 dx + \frac{B\beta^3\eps}{2}\norm{\ptxxx\ueb(t,\cdot)}^2_{L^2(\R)}\\
&\qquad\le \frac{3\beta}{2B\eps}\norm{\ueb}^2_{L^{\infty}((0,T)\times\R)}\norm{\ueb(t,\cdot)\px\ueb(t,\cdot)}^2_{L^2(\R)}\\
&\qquad\quad + \frac{B\beta^3\eps}{2}\norm{\ptxxx\ueb(t,\cdot)}^2_{L^2(\R)}\\
&\qquad\le \frac{C_{0}D\eps}{B}\norm{\ueb(t,\cdot)\px\ueb(t,\cdot)}^2_{L^2(\R)}+ \frac{B\beta^3\eps}{2}\norm{\ptxxx\ueb(t,\cdot)}^2_{L^2(\R)}.
\end{align*}
Then, it follows from \eqref{eq:P11} that
\begin{equation}
\label{eq:P15}
\begin{split}
&\frac{d}{dt}\left( \frac{1}{4}\norm{\ueb(t,\cdot)}^4_{L^4(\R)}+ \frac{\left(A+C\right)\eps^2}{2}\norm{\px\ueb(t,\cdot)}^2_{L^2(\R)}  \right)\\
&\qquad\quad+\frac{d}{dt}\left(\frac{A\beta^2\eps^2}{2}\norm{\pxxx\ueb(t,\cdot)}^2_{L^2(\R)}+ \frac{E\beta^4}{2}\norm{\pxxxx\ueb(t,\cdot)}^2_{L^2(\R)}\right)\\
&\qquad\quad +\frac{B\beta\eps^2+E\beta^2}{2}\frac{d}{dt}\norm{\pxx\ueb(t,\cdot)}^2_{L^2(\R)} +B\beta\eps\norm{\ptx\ueb(t,\cdot)}^2_{L^2(\R)}\\
&\qquad\quad +\left(1-2C\right)C\eps\norm{\pt\ueb(t,\cdot)}^2_{L^2(\R)} +\frac{B\beta^3\eps}{2}\norm{\ptxxx\ueb(t,\cdot)}^2_{L^2(\R)}\\
&\qquad\quad +\left(\frac{C}{2} -8B^2\right)\beta^2\eps\norm{\ptxx\ueb(t,\cdot)}^2_{L^2(\R)}+\left(E-1\right)\frac{\beta^2\eps}{2}\norm{\pxxx\ueb(t,\cdot)}^2_{L^2(\R)}\\
&\qquad\quad +\left(A -2A^2-C^2_0D^2\right)\eps^3\norm{\pxx\ueb(t,\cdot)}^2_{L^2(\R)}\\
&\qquad\quad+\left(\frac{3}{2}-\frac{C_0 D}{B}\right)\eps\norm{\ueb(t,\cdot)\px\ueb(t,\cdot)}^2_{L^2(\R)}
\le C_{0}\eps\norm{\px\ueb(t,\cdot)}^2_{L^2(\R)}.
\end{split}
\end{equation}
We search $A,\,B,\,C,\,E$ such that
\begin{equation*}
\begin{cases}
\displaystyle 1-2C >0,\\
\displaystyle \frac{C}{2} -8B^2 >0,\\
\displaystyle E-1>0,\\
\displaystyle A -2A^2-C^2_0D^2>0,\\
\displaystyle \frac{3}{2}-\frac{C_0 D}{B}>0,
\end{cases}
\end{equation*}
that is
\begin{equation}
\label{eq:P16}
\begin{cases}
\displaystyle C <\frac{1}{2},\\
\displaystyle B^2 < \frac{C}{16},\\
\displaystyle E>1,\\
\displaystyle 2A^2 -A +C^2_0D^2<0,\\
\displaystyle D<\frac{3B}{2C_{0}}.
\end{cases}
\end{equation}
We choose
\begin{equation}
\label{eq:P17}
C=\frac{1}{4},\quad E=2.
\end{equation}
It follows from the second inequality of \eqref{eq:P16}, and \eqref{eq:P17} that
\begin{equation*}
B<\frac{1}{8}.
\end{equation*}
Hence, we can choose
\begin{equation}
\label{eq:P18}
B=\frac{1}{9}.
\end{equation}
Substituting \eqref{eq:P18} in the fifth inequality of \eqref{eq:P16}, we have
\begin{equation}
\label{eq:P19}
D<\frac{1}{6C_0}.
\end{equation}
The fourth inequality admits solution when
\begin{equation}
\label{eq:P20}
D<\frac{2\sqrt{2}}{8C_0}.
\end{equation}
It follows from \eqref{eq:P19} and \eqref{eq:P20} that
\begin{equation}
\label{eq:P21}
D< \min\left\{\frac{1}{6C_0}, \frac{2\sqrt{2}}{8C_0}\right\}=\frac{1}{6C_0}.
\end{equation}
Therefore, from \eqref{eq:P16} and \eqref{eq:P21}, there exist $0<A_1<A_2$ such that
\begin{equation}
\label{eq:w1}
\quad 0<A_1<A<A_2.
\end{equation}
Substituting \eqref{eq:P17}, \eqref{eq:P18}, and \eqref{eq:P21} in \eqref{eq:P15}, from \eqref{eq:w1},  we get
\begin{align*}
&\frac{d}{dt}\left( \frac{1}{4}\norm{\ueb(t,\cdot)}^4_{L^4(\R)}+ \frac{\left(4A+1\right)\eps^2}{8}\norm{\px\ueb(t,\cdot)}^2_{L^2(\R)}  \right)\\
&\qquad\quad+\frac{d}{dt}\left(\frac{A\beta^2\eps^2}{2}\norm{\pxxx\ueb(t,\cdot)}^2_{L^2(\R)}+ \beta^4\norm{\pxxxx\ueb(t,\cdot)}^2_{L^2(\R)}\right)\\
&\qquad\quad +\frac{\beta\eps^2+18\beta^2}{18}\frac{d}{dt}\norm{\pxx\ueb(t,\cdot)}^2_{L^2(\R)} +\frac{\beta\eps}{9}\norm{\ptx\ueb(t,\cdot)}^2_{L^2(\R)}\\
&\qquad\quad +\frac{\eps}{8}\norm{\pt\ueb(t,\cdot)}^2_{L^2(\R)} +\frac{\beta^3\eps}{18}\norm{\ptxxx\ueb(t,\cdot)}^2_{L^2(\R)}\\
&\qquad\quad +\frac{73\beta^2\eps}{648}\norm{\ptxx\ueb(t,\cdot)}^2_{L^2(\R)}+\frac{\beta^2\eps}{2}\norm{\pxxx\ueb(t,\cdot)}^2_{L^2(\R)}\\
&\qquad\quad +K_2\eps^3\norm{\pxx\ueb(t,\cdot)}^2_{L^2(\R)}+K_2\eps\norm{\ueb(t,\cdot)\px\ueb(t,\cdot)}^2_{L^2(\R)}\\
&\qquad\le C_{0}\eps\norm{\px\ueb(t,\cdot)}^2_{L^2(\R)},
\end{align*}
for some $K_1,\,K_2>0$.

An integration on $(0,t)$, \eqref{eq:l-2-u1}, and \eqref{eq:u0eps-14} give
\begin{align*}
&\frac{1}{4}\norm{\ueb(t,\cdot)}^4_{L^4(\R)}+ \frac{\left(4A+1\right)\eps^2}{8}\norm{\px\ueb(t,\cdot)}^2_{L^2(\R)}\\
&\qquad\quad+\frac{A\beta^2\eps^2}{2}\norm{\pxxx\ueb(t,\cdot)}^2_{L^2(\R)}+ \beta^4\norm{\pxxxx\ueb(t,\cdot)}^2_{L^2(\R)}\\
&\qquad\quad +\frac{\beta\eps^2+18\beta^2}{18}\norm{\pxx\ueb(t,\cdot)}^2_{L^2(\R)} +\frac{\beta\eps}{9}\int_{0}^{t}\norm{\ptx\ueb(s,\cdot)}^2_{L^2(\R)}ds\\
&\qquad\quad +\frac{\eps}{8}\int_{0}^{t}\norm{\pt\ueb(s,\cdot)}^2_{L^2(\R)}ds +\frac{\beta^3\eps}{18}\int_{0}^{t}\norm{\ptxxx\ueb(s,\cdot)}^2_{L^2(\R)}ds\\
&\qquad\quad +\frac{73\beta^2\eps}{648}\int_{0}^{t}\norm{\ptxx\ueb(s,\cdot)}^2_{L^2(\R)}ds+\frac{\beta^2\eps}{2}\int_{0}^{t}\norm{\pxxx\ueb(s,\cdot)}^2_{L^2(\R)}ds\\
&\qquad\quad +K_2\eps^3\int_{0}^{t}\norm{\pxx\ueb(t,\cdot)}^2_{L^2(\R)}ds+K_2\eps\int_{0}^{t}\norm{\ueb(s,\cdot)\px\ueb(s,\cdot)}^2_{L^2(\R)}ds\\
&\qquad\le C_{0} +C_{0}\eps\int_{0}^{t}\norm{\px\ueb(s,\cdot)}^2_{L^2(\R)}ds \le C_0.
\end{align*}
Hence,
\begin{align*}
\norm{\ueb(t,\cdot)}_{L^4(\R)}\le & C_0,\\
\eps\norm{\px\ueb(t,\cdot)}_{L^2(\R)}\le &C_0,\\
\beta\eps\norm{\pxxx\ueb(t,\cdot)}_{L^2(\R)}\le &C_0,\\
\beta^2\norm{\pxxxx\ueb(t,\cdot)}_{L^2(\R)}\le &C_0,\\
\beta^{\frac{1}{2}}\eps\norm{\pxx\ueb(t,\cdot)}_{L^2(\R)}\le &C_0,\\
\beta\norm{\pxx\ueb(t,\cdot)}_{L^2(\R)}\le &C_0,\\
\beta\eps\int_{0}^{t}\norm{\ptx\ueb(s,\cdot)}^2_{L^2(\R)}ds\le &C_0,\\
\eps\int_{0}^{t}\norm{\pt\ueb(s,\cdot)}^2_{L^2(\R)}ds\le &C_0,\\
\beta^3\eps\int_{0}^{t}\norm{\ptxxx\ueb(s,\cdot)}^2_{L^2(\R)}ds\le &C_0,\\
\beta^2\eps\int_{0}^{t}\norm{\ptxx\ueb(s,\cdot)}^2_{L^2(\R)}ds\le &C_0,\\
\beta^2\eps\int_{0}^{t}\norm{\pxxx\ueb(s,\cdot)}^2_{L^2(\R)}ds\le &C_0,\\
\eps^3\int_{0}^{t}\norm{\pxx\ueb(t,\cdot)}^2_{L^2(\R)}ds\le &C_0,\\
\eps\int_{0}^{t}\norm{\ueb(s,\cdot)\px\ueb(s,\cdot)}^2_{L^2(\R)}ds\le &C_0,
\end{align*}
for every $0<t<T$.
\end{proof}
We are ready for the proof of Theorem \ref{th:main-13}.
\begin{proof}[Proof of Theorem \ref{th:main-13}.]
Let us consider a compactly supported entropy--entropy flux pair $(\eta,\,q)$. Multiplying \eqref{eq:Ro-eps-beta} by $\eta'(\ueb)$, we have
\begin{align*}
\pt\eta(\ueb) + \px q(\ueb) =&\eps \eta'(\ueb) \pxx\ueb-\beta\pxxx\ueb -\beta^2\eta'(\ueb)\ptxxxx\ueb \\
=& I_{1,\,\eps,\,\beta}+I_{2,\,\eps,\,\beta}+ I_{3,\,\eps,\,\beta} + I_{4,\,\eps,\,\beta}+ I_{5,\,\eps,\,\beta} + I_{6,\,\eps,\,\beta}
\end{align*}
where $I_{1,\,\eps,\,\beta},\,I_{2,\,\eps,\,\beta},\, I_{3,\,\eps,\,\beta},\, I_{4,\,\eps,\,\beta},\,I_{5,\,\eps,\,\beta},\,I_{6,\,\eps,\,\beta}$ are defined in \eqref{eq:12000}.

As in \cite[Theorem $3.1$]{Cd6}, we obtain that $I_{1,\,\eps,\,\beta}\to 0$ in $H^{-1}((0,T)\times\R)$, $\{I_{2,\,\eps,\,\beta}\}_{\eps,\beta>0}$ is bounded in $L^1((0,T)\times\R)$, $I_{4,\,\eps,\,\beta}\to 0$ in $H^{-1}((0,T)\times\R)$, $I_{5,\,\eps,\,\beta}\to 0$ in $L^1((0,T)\times\R)$, while as in \cite[Theorem $2.1$]{Cd5} $I_{3,\,\eps,\,\beta}\to 0$ in $H^{-1}((0,T)\times\R)$, and, $I_{4,\,\eps,\,\beta}\to 0$ in $L^1((0,T)\times\R)$

Arguing in \cite[Theorem $2.1$]{Cd5}, we have \eqref{eq:u-entro-sol-12}.
\end{proof}

\appendix
\section{The Benjamin-Bona-Mahony equation}\label{appen1}
In this appendix, we consider The Benjamin-Bona-Mahony equation
\begin{equation}
\label{eq:BBM}
\pt u +u\px u -\beta\ptxx u =0.
\end{equation}
We augment \eqref{eq:BBM} with the initial condition
\begin{equation}
u(0,x)=u_{0}(x),
\end{equation}
on which we assume \eqref{eq:uo-l2}
We study the dispersion-diffusion limit for \eqref{eq:BBM}. Therefore, we fix two small numbers $\eps,\,\beta$ and consider the following third order problem
\begin{equation}
\label{eq:AB2}
\begin{cases}
\pt\ueb+ \ueb\px \ueb -\beta\ptxx\ueb =\eps\pxx\ueb, &\qquad t>0, \ x\in\R ,\\
\ueb(0,x)=u_{\eps,\beta,0}(x), &\qquad x\in\R,
\end{cases}
\end{equation}
where $u_{\eps,\beta,0}$ is a $C^\infty$ approximation of $u_{0}$ such that
\begin{equation}
\begin{split}
\label{eq:AB3}
&u_{\eps,\,\beta,\,0} \to u_{0} \quad  \textrm{in $L^{p}_{loc}(\R)$, $1\le p < 2$, as $\eps,\,\beta \to 0$,}\\
&\norm{u_{\eps,\beta, 0}}^2_{L^2(\R)}+\left(\beta+\beta^{\frac{1}{2}}\right)\norm{\px u_{\eps,\beta,0}}^2_{L^2(\R)}\le C_0,\quad \eps,\beta >0,  \\
&\left(\beta^{\frac{3}{2}}+\beta\eps^2\right)\norm{\pxx u_{\eps,\beta,0}}^2_{L^2(\R)}\le C_0,\quad \eps,\beta >0,
\end{split}
\end{equation}
and $C_0$ is a constant independent on $\eps$ and $\beta$.

The main result of this section is the following theorem.
\begin{theorem}
\label{th:main-A2}
Assume that \eqref{eq:uo-l2} and \eqref{eq:AB3} hold. If \eqref{eq:beta-eps-2} holds,
then, there exist two sequences $\{\eps_{n}\}_{n\in\N}$, $\{\beta_{n}\}_{n\in\N}$, with $\eps_n, \beta_n \to 0$, and a limit function
\begin{equation*}
u\in L^{\infty}(\R^{+}; L^2(\R)),
\end{equation*}
such that
\begin{itemize}
\item[$i)$] $u_{\eps_n, \beta_n}\to u$  strongly in $L^{p}_{loc}(\R^{+}\times\R)$, for each $1\le p <2$,
\item[$ii)$] $u$ a distributional solution of \eqref{eq:BU}.
\end{itemize}
Moreover, if \eqref{eq:beta-eps-4} holds
\begin{itemize}
\item[$iii)$] $u$ is the unique entropy solution of \eqref{eq:BU}.
\end{itemize}
\end{theorem}
Let us prove some a priori estimates on $\ueb$, denoting with $C_0$ the constants which depend only on the initial data.

Arguing as \cite{SC}, we have the following result
\begin{lemma}\label{lm:N23}
For each $t>0$,
\begin{equation}
\label{eq:AB31}
\norm{\ueb(t,\cdot)}^2_{L^2(\R)} + \beta\norm{\px\ueb(t,\cdot)}^2_{L^2(\R)} + 2\eps\int_{0}^{t}\norm{\px\ueb(t,\cdot)}^2_{L^2(\R)} \le C_0.
\end{equation}
Moreover,
\begin{equation}
\label{eq:AB2*}
\norm{\ueb(t,\cdot)}_{L^{\infty}(\R)}\le C_{0}\beta^{-\frac{1}{4}}.
\end{equation}
\end{lemma}
\begin{lemma}
Assume \eqref{eq:beta-eps-2}. For each $t>0$,
\begin{equation}
\label{eq:BC3}
\begin{split}
&\beta\norm{\px\ueb(t,\cdot)}^2_{L^2(\R)}+\frac{2\beta^2 + \beta^{\frac{3}{2}}\eps^2}{2} \norm{\pxx\ueb(t,\cdot)}^2_{L^2(\R)}\\
&\qquad\quad + \frac{3\beta\eps}{2}\int_{0}^{t}\norm{\pxx\ueb(s,\cdot)}^2_{L^2(\R)} +\frac{\beta^{\frac{5}{2}}\eps}{2}\int_{0}^{t}\norm{\ptxx\ueb(s,\cdot)}^2_{L^2(\R)}ds\\
&\qquad\quad +\beta^{\frac{3}{2}}\eps\norm{\ptx\ueb(t,\cdot)}^2_{L^2(\R)}\le C_0.
\end{split}
\end{equation}
\end{lemma}
\begin{proof}
Let $t>0$. Multiplying \eqref{eq:AB2} by $-2\beta^{\frac{1}{2}}\pxx\ueb-\beta\eps\ptxx\ueb$, we have
\begin{equation}
\label{eq:AB3*}
\begin{split}
\left(-2\beta^{\frac{1}{2}}\pxx\ueb-\beta\eps\ptxx\ueb\right)\pt\ueb&+ \left(-2\beta^{\frac{1}{2}}\pxx\ueb-\beta\eps\ptxx\ueb\right)\ueb\px\ueb\\
&-\beta\left(-2\beta^{\frac{1}{2}}\pxx\ueb-\beta\eps\ptxx\ueb\right)\ptxx\ueb\\
=&\eps\left(-2\beta^{\frac{1}{2}}\pxx\ueb-\beta\eps\ptxx\ueb\right)\pxx\ueb.
\end{split}
\end{equation}
Since
\begin{align*}
&\int_{\R} \left(-2\beta^{\frac{1}{2}}\pxx\ueb-\beta\eps\ptxx\ueb\right)\pt\ueb dx\\
&\qquad =\beta^{\frac{1}{2}}\frac{d}{dt}\norm{\px\ueb(t,\cdot)}^2_{L^2(\R)} +\beta\eps\norm{\ptx\ueb(t,\cdot)}^2_{L^2(\R)},\\
&-\beta\int_{\R}\left(-2\beta^{\frac{1}{2}}\pxx\ueb-\beta\eps\ptxx\ueb\right)\ptxx\ueb dx \\
&\qquad=\beta^{\frac{3}{2}}\frac{d}{dt}\norm{\pxx\ueb(t,\cdot)}^2_{L^2(\R)}+\beta^2\eps\norm{\ptxx\ueb(t,\cdot)}^2_{L^2(\R)},\\
&\eps\int_{\R}\left(-2\beta^{\frac{1}{2}}\pxx\ueb-\beta\eps\ptxx\ueb\right)\pxx\ueb dx\\
&\qquad = -2\beta^{\frac{1}{2}}\eps\norm{\pxx\ueb(t,\cdot)}^2_{L^2(\R)}-\frac{\beta\eps^2}{2}\frac{d}{dt}\norm{\pxx\ueb(t,\cdot)}^2_{L^2(\R)},
\end{align*}
integrating \eqref{eq:AB3*} on $\R$, we get
\begin{equation}
\label{eq:AB4}
\begin{split}
&\frac{d}{dt}\left(\beta^{\frac{1}{2}}\norm{\px\ueb(t,\cdot)}^2_{L^2(\R)}+\frac{2\beta^{\frac{3}{2}} + \beta\eps^2}{2} \norm{\pxx\ueb(t,\cdot)}^2_{L^2(\R)}\right)\\
&\qquad +2\beta^{\frac{1}{2}}\eps\norm{\pxx\ueb(t,\cdot)}^2_{L^2(\R)} +\beta^2\eps\norm{\ptxx\ueb(t,\cdot)}^2_{L^2(\R)}\\
&\qquad +\beta\eps\norm{\ptx\ueb(t,\cdot)}^2_{L^2(\R)}\\
&\qquad= 2\beta^{\frac{1}{2}}\int_{\R}\ueb\px\ueb\pxx\ueb dx -\beta\eps\int_{\R}\ueb\px\ueb\ptxx\ueb dx.
\end{split}
\end{equation}
Due to \eqref{eq:beta-eps-2}, \eqref{eq:AB2*}, and the Young inequality,
\begin{equation}
\label{eq:AB5}
\begin{split}
&2\beta^{\frac{1}{2}}\int_{\R}\vert\ueb\px\ueb\vert\vert\pxx\ueb\vert dx =\beta^{\frac{1}{2}}\int_{\R}\left\vert\frac{2\ueb\px\ueb}{\eps^{\frac{1}{2}}}\right\vert\left\vert \eps^{\frac{1}{2}}\pxx\ueb\right\vert dx \\
&\qquad \le \frac{2\beta^{\frac{1}{2}}}{\eps}\int_{\R}\ueb^2(\px\ueb)^2 dx +\frac{\beta^{\frac{1}{2}}\eps}{2}\norm{\pxx\ueb(t,\cdot)}^2_{L^2(\R)}\\
&\qquad \le C_{0}\eps\int_{\R}\ueb^2(\px\ueb)^2 dx + \frac{\beta^{\frac{1}{2}}\eps}{2}\norm{\pxx\ueb(t,\cdot)}^2_{L^2(\R)}\\
&\qquad \le C_{0}\eps\norm{\ueb(t,\cdot)}^2_{L^{\infty}(\R)}\norm{\px\ueb(t,\cdot)}^2_{L^2(\R)} + \frac{\beta^{\frac{1}{2}}\eps}{2}\norm{\pxx\ueb(t,\cdot)}^2_{L^2(\R)}\\
&\qquad\le\frac{C_{0}\eps}{\beta^{\frac{1}{2}}}\norm{\px\ueb(t,\cdot)}^2_{L^2(\R)} + \frac{\beta^{\frac{1}{2}}\eps}{2}\norm{\pxx\ueb(t,\cdot)}^2_{L^2(\R)}.
\end{split}
\end{equation}
Thanks to \eqref{eq:AB2*}, and the Young inequality,
\begin{equation}
\label{eq:AB6}
\begin{split}
&\beta\eps\int_{\R}\vert\ueb\px\ueb\vert\vert\ptxx\ueb\vert dx= \eps\int_{\R}\left\vert\ueb\px\ueb\right\vert\left\vert\beta\ptxx\ueb\right\vert dx\\
&\qquad\le\frac{\eps}{2}\int_{\R}\ueb^2(\px\ueb)^2 dx + \frac{\beta\eps}{2}\norm{\ptxx\ueb(t,\cdot)}^2_{L^2(\R)}\\
&\qquad\le\frac{\eps}{2}\norm{\ueb(t,\cdot)}^2_{L^{\infty}(\R)}\norm{\px\ueb(t,\cdot)}^2_{L^2(\R)} + \frac{\beta\eps}{2}\norm{\ptxx\ueb(t,\cdot)}^2_{L^2(\R)}\\
&\qquad\le \frac{\eps}{2\beta^{\frac{1}{2}}}\norm{\px\ueb(t,\cdot)}^2_{L^2(\R)} + \frac{\beta\eps}{2}\norm{\ptxx\ueb(t,\cdot)}^2_{L^2(\R)}.
\end{split}
\end{equation}
It follows from \eqref{eq:AB4}, \eqref{eq:AB5}, and \eqref{eq:AB6} that
\begin{align*}
&\frac{d}{dt}\left(\beta^{\frac{1}{2}}\norm{\px\ueb(t,\cdot)}^2_{L^2(\R)}+\frac{2\beta^{\frac{3}{2}} + \beta\eps^2}{2} \norm{\pxx\ueb(t,\cdot)}^2_{L^2(\R)}\right)\\
&\qquad +\frac{3\beta^{\frac{1}{2}}\eps}{2}\norm{\pxx\ueb(t,\cdot)}^2_{L^2(\R)} +\frac{\beta^2\eps}{2}\norm{\ptxx\ueb(t,\cdot)}^2_{L^2(\R)}\\
&\qquad+\beta\eps\norm{\ptx\ueb(t,\cdot)}^2_{L^2(\R)}\le \frac{C_{0}\eps}{\beta^{\frac{1}{2}}}\norm{\px\ueb(t,\cdot)}^2_{L^2(\R)}.
\end{align*}
Hence,
\begin{align*}
&\frac{d}{dt}\left(\beta\norm{\px\ueb(t,\cdot)}^2_{L^2(\R)}+\frac{2\beta^2 + \beta^{\frac{3}{2}}\eps^2}{2} \norm{\pxx\ueb(t,\cdot)}^2_{L^2(\R)}\right)\\
&\qquad +\frac{3\beta\eps}{2}\norm{\pxx\ueb(t,\cdot)}^2_{L^2(\R)} +\frac{\beta^{\frac{5}{2}}\eps}{2}\norm{\ptxx\ueb(t,\cdot)}^2_{L^2(\R)}\\
&\qquad +\beta^{\frac{3}{2}}\eps\norm{\ptx\ueb(t,\cdot)}^2_{L^2(\R)}\le C_0\eps\norm{\px\ueb(t,\cdot)}^2_{L^2(\R)}.
\end{align*}
An integration on $(0,t)$ and \eqref{eq:AB31} give
\begin{align*}
&\beta\norm{\px\ueb(t,\cdot)}^2_{L^2(\R)}+\frac{2\beta^2 + \beta^{\frac{3}{2}}\eps^2}{2} \norm{\pxx\ueb(t,\cdot)}^2_{L^2(\R)}\\
&\qquad\quad + \frac{3\beta\eps}{2}\int_{0}^{t}\norm{\pxx\ueb(s,\cdot)}^2_{L^2(\R)} +\frac{\beta^{\frac{5}{2}}\eps}{2}\int_{0}^{t}\norm{\ptxx\ueb(s,\cdot)}^2_{L^2(\R)}ds\\
&\qquad+\beta^{\frac{3}{2}}\eps\norm{\ptx\ueb(t,\cdot)}^2_{L^2(\R)}\le C_0+ C_0\eps\int_{0}^{t}\norm{\px\ueb(s,\cdot)}^2_{L^2(\R)}ds \le C_0,
\end{align*}
that is \eqref{eq:BC3}.
\end{proof}
We continue by proving the following result
\begin{lemma}\label{lm:9000}
Assume that \eqref{eq:uo-l2}, \eqref{eq:beta-eps-2}, and \eqref{eq:AB3} hold. Then, for any compactly
supported entropy--entropy flux pair $(\eta, \,q)$, there exist two sequences $\{\eps_{n}\}_{n\in\N},\,\{\beta_{n}\}_{n\in\N}$, with $\eps_n,\,\beta_n\to0$, and a limit function
\begin{equation*}
u\in L^{\infty}(\R^{+};L^2(\R)),
\end{equation*}
such that
\eqref{eq:con-u-1} holds and
\begin{equation}
\label{eq:A40}
\textrm{$u$ is a distributional solution of \eqref{eq:BU}}.
\end{equation}
\end{lemma}
\begin{proof}
Let us consider a compactly supported entropy--entropy flux pair $(\eta, q)$. Multiplying \eqref{eq:AB2} by $\eta'(\ueb)$, we have
\begin{align*}
\pt\eta(\ueb) + \px q(\ueb) =&\eps \eta'(\ueb) \pxx\ueb +\beta\eta'(\ueb)\ptxx\ueb \\
=& I_{1,\,\eps,\,\beta}+I_{2,\,\eps,\,\beta}+ I_{3,\,\eps,\,\beta} + I_{4,\,\eps,\,\beta},
\end{align*}
where
\begin{equation}
\begin{split}
\label{eq:1200013}
I_{1,\,\eps,\,\beta}&=\px(\eps\eta'(\ueb)\px\ueb),\\
I_{2,\,\eps,\,\beta}&= -\eps\eta''(\ueb)(\px\ueb)^2,\\
I_{3,\,\eps,\,\beta}&= \px(\beta\eta'(\ueb)\ptx\ueb),\\
I_{4,\,\eps,\,\beta}&= -\beta\eta''(\ueb)\px\ueb\ptx\ueb.
\end{split}
\end{equation}
Fix $T>0$. Arguing in \cite[Lemma $3.2$]{Cd2}, we have that $I_{1,\,\eps,\,\beta}\to0$ in $H^{-1}((0,T) \times\R)$, and $\{I_{2,\,\eps,\,\beta}\}_{\eps,\beta >0}$ is bounded in $L^1((0,T)\times\R)$.\\
We claim that
\begin{equation*}
I_{3,\,\eps,\,\beta}\to0 \quad \text{in $H^{-1}((0,T) \times\R),\,T>0,$ as $\eps\to 0$.}
\end{equation*}
By \eqref{eq:beta-eps-2} and \eqref{eq:BC3},
\begin{align*}
&\norm{ \beta\eta'(\ueb)\ptx\ueb}^2_{L^2((0,T)\times\R)}\\
&\qquad\le \beta^2 \norm{\eta'}_{L^{\infty}(\R)}\norm{\ptx\ueb}^2_{L^2((0,T)\times\R)}\\
&\qquad= \norm{\eta'}_{L^{\infty}(\R)}\frac{\beta^2\eps}{\eps}\norm{\ptx\ueb}^2_{L^2((0,T)\times\R)}\\
&\qquad=\norm{\eta'}_{L^{\infty}(\R)}\frac{\beta^{\frac{1}{2}}\beta^{\frac{3}{2}}\eps}{\eps}\norm{\ptx\ueb}^2_{L^2((0,T)\times\R)}
\le C_{0}\norm{\eta'}_{L^{\infty}(\R)}\eps\to 0.
\end{align*}
Let us show that
\begin{equation*}
I_{4,\,\eps,\,\beta}\quad \text{is bounded in $L^1((0,T) \times\R),\,T>0,$.}
\end{equation*}
Thanks to \eqref{eq:beta-eps-2}, \eqref{eq:AB31}, \eqref{eq:BC3}, and the H\"older inequality,
\begin{align*}
&\norm{\beta\eta''(\ueb)\px\ueb\ptx\ueb}_{L^1((0,T)\times\R)}\\
&\qquad\le\beta\norm{\eta''}_{L^{\infty}(\R)}\int_{0}^{T}\!\!\!\int_{\R}\vert\px\ueb\ptx\ueb\vert dsdx\\
&\qquad=\norm{\eta''}_{L^{\infty}(\R)}\frac{\beta^{\frac{1}{4}}\beta^{\frac{3}{4}}\eps}{\eps}\norm{\px\ueb}_{L^2((0,T)\times\R)}\norm{\ptx\ueb}_{L^2((0,T)\times\R)}\\
&\qquad\le C_{0}\norm{\eta''}_{L^{\infty}(\R)}\frac{\beta^{\frac{1}{4}}}{\eps}\le C_{0}\norm{\eta''}_{L^{\infty}(\R)}.
\end{align*}
Arguing as in \cite{SC}, we have \eqref{eq:A40}.
\end{proof}
\begin{lemma}\label{eq:10034}
Assume \eqref{eq:uo-l2}, \eqref{eq:beta-eps-4},  and \eqref{eq:AB3} hold. Then, for any compactly
supported entropy--entropy flux pair $(\eta, \,q)$, there exist two sequences $\{\eps_{n}\}_{n\in\N},\,\{\beta_{n}\}_{n\in\N}$, with $\eps_n,\,\beta_n\to0$, and a limit function
\begin{equation*}
u\in L^{\infty}(\R^{+};L^2(\R)),
\end{equation*}
such that
\eqref{eq:con-u-1} and \eqref{eq:u-entro-sol-12} hold.
\end{lemma}
\begin{proof}
Let us consider a compactly supported entropy--entropy flux pair $(\eta, q)$. Multiplying \eqref{eq:AB2} by $\eta'(\ueb)$, we have
\begin{align*}
\pt\eta(\ueb) + \px q(\ueb) =&\eps \eta'(\ueb) \pxx\ueb +\beta\eta'(\ueb)\ptxx\ueb  \\
=& I_{1,\,\eps,\,\beta}+I_{2,\,\eps,\,\beta}+ I_{3,\,\eps,\,\beta} + I_{4,\,\eps,\,\beta},
\end{align*}
where $I_{1,\,\eps,\,\beta},\,I_{2,\,\eps,\,\beta},\, I_{3,\,\eps,\,\beta},\, I_{4,\,\eps,\,\beta}$ are defined in \eqref{eq:1200013}.

As in Lemma \ref{lm:259}, we have that $I_{1,\,\eps,\,\beta},\,I_{3,\,\eps,\,\beta}   \to 0$ in $H^{-1}((0,T)\times\R)$, $\{ I_{2,\,\eps,\,\beta}\}_{\eps,\beta>0}$ is bounded in $L^1((0,T)\times\R)$, while $I_{4,\,\eps,\,\beta}\to0$ in $L^1((0,T)\times\R)$.

Arguing as in \cite{LN}, we have \eqref{eq:u-entro-sol-12}.
\end{proof}

\begin{proof}[Proof of Theorem \ref{th:main-A2}]
Theorem \ref{th:main-A2} follows from Lemmas \ref{lm:9000} and  \ref{eq:10034}.
\end{proof}


\begin{thebibliography}{40}

\bibitem{AB}
{\sc M. Antonova and A. Biswas.}
\newblock Adiabatic parameter dynamics of perturbed solitary waves.
\newblock{\em Communications in Nonlinear Science and Numerical Simulation}, 14:734--748, 2009.

\bibitem{Ba}
{\sc A. R. Bahadır}
\newblock Exponential finite--difference method applied to Korteweg--de Vries equation for small times.
\newblock {\em Applied Mathematics and Computation}, 160(3):675--682, 2005.


\bibitem{BTL}
{\sc A. Biswas, H. Triki and M. Labidi.}
\newblock Bright and dark solitons of the Rosenau-Kawahara equation with power law nonlinearity.
\newblock{\em Physics of Wave Phenomena}, 19:24--29, 2011.

\bibitem{B}
{\sc J. Boyd.}
\newblock Ostrovsky and HunterÕs generic wave equation for weakly dispersive waves: matched
asymptotic and pseudospectral study of the paraboloidal travelling waves (corner and near-corner waves).
\newblock {\em Euro. Jnl. of Appl. Math.}, 16(1):65--81, 2005.

\bibitem{CH1}
{\sc S. K. Chung.}
\newblock Finite difference approximate solutions for the Rosenau equation.
\newblock{\em Applicable Analysis}, vol. 69, no. 1--2, pp. 149--156, 1998.



\bibitem{CHH}
{\sc S. K. Chung and S.N.Ha.}
\newblock Finite element Galerkin solutions for the Rosenau equation.
\newblock {\em Applicable Analysis}, vol. 54, no. 1--2, pp. 39--56, 1994.

\bibitem{CHP}
{\sc S. K. Chung and A. K. Pani.}
\newblock Numerical methods for the Rosenau equation.
\newblock {\em Applicable Analysis}, vol. 77, no. 3--4, pp. 351--369, 2001.


\bibitem{Cd5}
{\sc G. M. Coclite and L. di Ruvo.}
\newblock A singular limit problem for the Rosenau-Korteweg-de Vries regularized long wave and Rosenau Korteweg-de Vries equation.
\newblock Submitted.


\bibitem{Cd6}
{\sc G. M. Coclite and L. di Ruvo.}
\newblock A singular limit problem for the Rosenau equation.
\newblock Submitted.

\bibitem{Cd}
{\sc G. M. Coclite and L. di Ruvo.}
\newblock A singular limit problem for conservation laws realted to the Kudryashov-Sinelshchikov equation.
\newblock Submitted.

\bibitem{Cd1}
{\sc G.~M. Coclite and L. di Ruvo.}
\newblock Oleinik type estimate for the Ostrovsky-Hunter equation.
\newblock {\em J. Math. Anal. Appl.}, 423:162--190, 2015.

\bibitem{Cd2}
{\sc G.~M. Coclite and L. di Ruvo.}
\newblock Convergence of the Ostrovsky Equation to the Ostrovsky-Hunter One.
\newblock {\em J. Differential Equations}, 256:3245--3277, 2014.

\bibitem{CdK}
{\sc G. M. Coclite, L. di Ruvo, and K. H. Karlsen}.
\newblock Some wellposedness results for the Ostrovsky-Hunter Equation.
\newblock {\em Hyperbolic conservation laws and related analysis with applications}, 143-159, Springer Proc. Math. Stat., 49, Springer, Heidelberg, 2014.


\bibitem{CdREM}
{\sc G. M. Coclite, L. di Ruvo, J. Ernest, and S. Mishra.}
\newblock Convergence of vanishing capillarity approximations for scalar conservation laws with discontinuous fluxes.
\newblock {\em Netw. Heterog. Media}, 8(4):969--984, 2013.

\bibitem{CK}
{\sc G. ~M. Coclite and  K.~H. Karlsen.}
\newblock A singular limit problem for conservation laws related to the Camassa-Holm shallow water equation.
\newblock {\em Comm. Partial Differential Equations}, 31:1253--1272, 2006.

\bibitem{CRS}
{\sc A. Corli, C. Rohde, and V. Schleper.}
\newblock Parabolic approximations of diffusive-dispersive equations.
\newblock {\em J. Math. Anal. Appl.} 414:773--798, 2014.

\bibitem{CM}
{\sc Y. Cui and D.k. Mao}
\newblock Numerical method satisfying the first two conservation laws for the Korteweg-de Vries equation.
\newblock{\em J. of Computational Physics}, 227(1):376--399, 2007.

\bibitem{dR}
{\sc L. di Ruvo.}
\newblock Discontinuous solutions for the Ostrovsky--Hunter equation and two phase flows.
\newblock {\em Phd Thesis, University of Bari}, 2013.
\newblock{www.dm.uniba.it/home/dottorato/dottorato/tesi/}.


\bibitem{E}
{\sc A. Esfahani.}
\newblock Solitary wave solutions for generalized Rosenau-KdV equation.
\newblock {\em Communications in Theoretical Physics}, 55(3):396--398, 2011.




\bibitem{EMTYB}
{\sc G. Ebadi, A. Mojaver, H. Triki, A. Yildirim, and A. Biswas.}
\newblock Topological solitons and other solutions of the Rosenau-KdV equation with power law nonlinearity.
\newblock {\em Romanian J. of Physics}, 58:3--14, 2013.



\bibitem{HXH}
{\sc J. Hu, Y. Xu, and B. Hu.}
\newblock Conservative Linear Difference Scheme for Rosenau-KdV Equation.
\newblock {\em Adv. Math. Phys.}, 423718, 2013.

\bibitem{KL}
{\sc Y. D. Kim and H. Y. Lee.}
\newblock The convergence of finite element Galerkin solution for the Roseneau equation.
\newblock {\em The Korean Journal of Computational \& Applied Mathematics}, 5(1):171--180, 1998.



\bibitem{LB}
{\sc M. Labidi and A. Biswas.}
\newblock Application of He’s principles to Rosenau-Kawahara equation.
\newblock{\em Mathematics in Engineering, Science and Aerospace}, 2:183--197, 2011.

\bibitem{LN}
{\sc P. G. LeFloch and R. Natalini.}
\newblock Conservation laws with vanishing nonlinear diffusion and dispersion.
\newblock {\em  Nonlinear Anal. 36, no. 2, Ser. A: Theory Methods}, 212--230, 1992

\bibitem{MPC}
{\sc S. A.V.Manickam, A. K. Pani, and S. K.Chung.}
\newblock A second-order splitting combined with orthogonal cubic spline collocation method for the Rosenau equation.
\newblock {\em Numerical Methods for Part. Diff. Equations}, 14(6):695--716, 1998.




\bibitem{Murat:Hneg}
{\sc F.~Murat.}
\newblock L'injection du c\^one positif de ${H}\sp{-1}$\ dans ${W}\sp{-1,\,q}$\  est compacte pour tout $q<2$.
\newblock {\em J. Math. Pures Appl. (9)}, 60(3):309--322, 1981.



\bibitem{OAAK}
{\sc K. Omrani, F. Abidi, T. Achouri, and N. Khiari.}
\newblock A new conservative finite difference scheme for the Rosenau equation.
\newblock {\em Applied Mathematics and Computation}, vol. 201, no. 1--2, pp. 35--43, 2008.


\bibitem{OK}
{\sc S. \"Ozer and S. Kutluay.}
\newblock An analytical-numerical method for solving the Korteweg-de Vries equation.
\newblock {\em Applied Mathematics and Computation}, 164(3):789--797, 2005.

\bibitem{O}
{\sc L. A. Ostrovsky.}
\newblock Nonlinear internal waves in a rotating ocean.
\newblock {\em Okeanologia}, 18:181--191, 1978.


\bibitem{P}
{\sc M. A. Park.}
\newblock On the Rosenau equation.
\newblock {\em Matem\'atica Aplicada e Computacional},  9(2):145--152, 1990.

\bibitem{RAB}
{\sc P. Razborova, B. Ahmed, and A. Biswas.}
\newblock Solitons, shock waves and conservation laws of Rosenau-KdV-RLW equation with power law nonlinearity
\newblock {\em Appl. Math. Inform. Sci.}, 8:485--491, 2014.

\bibitem{RTB}
{\sc P. Razborova, H. Triki, and A. Biswas.}
\newblock Perturbation of dispersive shallow water waves.
\newblock{\em Ocean Engineering}, 63:1--7, 2013.

\bibitem{Ro1}
{\sc P. Rosenau.}
\newblock A quasi-continuous description of a nonlinear transmission line.
\newblock {\em Physica Scripta},34:827--829, 1986.

\bibitem{Ro2}
{\sc P. Rosenau.}
\newblock Dynamics of dense discrete systems.
\newblock{\em Progress of Theoretical Physics},  79:1028--1042, 1988.

\bibitem{SC}
{\sc M. E. Schonbek.}
\newblock {Convergence of solutions to nonlinear dispersive equations}
\newblock {\em Comm. Partial Differential Equations}, 7(8):959--1000, 1982.

\bibitem{ZZ}
{\sc M. Zheng and J. Zhou.}
\newblock An average linear difference scheme for the generalized Rosenau--KdV Equation.
\newblock{\em J. of Appl. Math.} vol.2014, pages 9, 2014.

\bibitem{ZUZO}
{\sc S. Zhu and J. Zhao.}
\newblock The alternating segment explicit--implicit scheme for the dispersive equation.
\newblock{\em Applied Mathematics Letters}, 14(6):657--662, 2001.


\bibitem{Z}
{\sc J. M. Zuo.}
\newblock Solitons and periodic solutions for the Rosenau-KdV and Rosenau-Kawahara equations.
\newblock {\em Applied Mathematics and Computation}, 215(2):835--840, 2009.

\bibitem{ZZZC}
{\sc J.M. Zuo, Y.M. Zhang, T.D. Zhang and F. Chang}.
\newblock A new conservative difference scheme for the generalized Rosenau-RLW equation.
\newblock {\em Boundary Value Problems}, 516260, 2010.


\end{thebibliography}
\end{document}